\documentclass[11pt,twoside, leqno]{article}

\usepackage{mathrsfs}
\usepackage{amssymb}
\usepackage{amsmath}
\usepackage{amsthm}
\usepackage{amsfonts}
\usepackage{extarrows}
\usepackage{xy}
\usepackage{array}
\xyoption{all}

\usepackage{color}
\usepackage[colorlinks, linkcolor=blue, anchorcolor=green, citecolor=blue]{hyperref}
\usepackage[pagewise]{lineno}

\allowdisplaybreaks

\renewcommand\hat{\widehat}

\def\bint{{\ifinner\rlap{\bf\kern.30em--}
\int\else\rlap{\bf\kern.35em--}\int\fi}\ignorespaces}

\def\sbint{{\ifinner\rlap{\bf\kern.32em--}
\hspace{0.078cm}\int\else\rlap{\bf\kern.45em--}\int\fi}\ignorespaces}

\newtheorem{theorem}{Theorem}[section]
\newtheorem{lemma}[theorem]{Lemma}

\theoremstyle{definition}

\newtheorem{definition}[theorem]{Definition}

\numberwithin{equation}{section}

\numberwithin{equation}{section}

\numberwithin{equation}{section}

\begin{document}

\arraycolsep=1pt
\title{\Large\bf  Gevrey regularity solution for initial data in Triebel-Lizorkin-Lorentz spaces via single norm defined by nonlinearity of frequency
\footnotetext{\hspace{-0.15cm}
\endgraf $^\#$\,Corresponding author
}}

\author{ Qixiang Yang$^{1}$, Hongwei Li$^{1,\#}$}
\date{  }
\maketitle

\vspace{-0.8cm}

\begin{center}
\begin{minipage}{13cm}\small
\begin{center}
\emph{ $^1$\,School of Mathematics and Statistics, Wuhan University, Wuhan }430072\emph{, China}
\end{center}

\begin{center}
\emph{ Email: qxyang@whu.edu.cn., }3368554687\emph{@qq.com}
\end{center}

\hrule
\vspace{\baselineskip}
{\noindent{\bf Abstract} \quad
The properties of solutions to Navier-Stokes equations, including well-posedness and Gevrey regularity, are a class of highly interesting problems.
We are inspired by the location result on Triebel-Lizorkin-Lorentz space of Hobus and Saal in 2019.
In order to overcome the difficulties they encountered when dealing with global well-posedness, we introduce the single norm iterative space ${^{m'}_{m} \dot{F}}^{\frac{n}{p}-1,q}_{p,r}$ and utilize tools such as the Fefferman-Stein inequality to investigate the properties of our iterative spaces. 
As a result, we establish the global well-posedness of Navier-Stokes equations in critical Triebel-Lizorkin-Lorentz space and obtain the Gevrey regularity of the mild solution.
Regarding that there're many regularity studies focused on Besov spaces, such as Bae-Biswas-Tadmor(2012) and Liu-Zhang (2024),
our Triebel-Lizorkin-Lorentz spaces contain more general initial value spaces, including part of Besov spaces and all of Triebel-Lizorkin spaces, etc..
Furthermore, compared with Germain-Pavlović-Staffilani (2007), our Gevrey estimation also implies spatial analyticity and is more convenient to unify the estimates of gradient of any order.

}

\vspace{\baselineskip}
{\bf Keywords}\quad 
Navier-Stokes equations, Triebel-Lizorkin-Lorentz space,  Global well-posedness, Meyer wavelets, Gevrey regularity, Single norm
\vspace{\baselineskip}

{\bf MSC(2020)} \quad 35Q30,76D03,42B35,46E30
\vspace{\baselineskip}
\hrule
\end{minipage}
\end{center}




\section{Motivations and Main Theorem}\label{intro}
\par  For $n\geq2$,
consider the Cauchy problem of the Navier-Stokes equations on half-space $\mathbb{R}^{1+n}_+=(0,\infty)\times\mathbb{R}^{n}$:
\begin{equation}\label{NS}
\begin{cases}
\partial_tu-\Delta u+(u\cdot\nabla) u-\nabla p=0,  & \texttt{ in } \;\mathbb{R}^{1+n}_+\,;\\
\nabla\cdot u=0,  & \texttt{ in } \;\mathbb{R}^{1+n}_+\,;\\
u\big|_{t=0}=f,   & \texttt{ in } \;\mathbb{R}^{n}\,.
\end{cases}
\end{equation}
Denote Leray projector by $\mathbb{P}$. The divergence zero gives
\begin{equation}\label{NS1}
\begin{cases}
\partial_tu-\Delta u+\mathbb{P}\nabla\cdot(u\otimes u)=0,  & \texttt{ in } \;\mathbb{R}^{1+n}_+\,;\\
\nabla\cdot u=0,  & \texttt{ in } \;\mathbb{R}^{1+n}_+\,;\\
u\big|_{t=0}=f,   & \texttt{ in } \;\mathbb{R}^{n}\,.
\end{cases}
\end{equation}
\par We write (\ref{NS1}) in the form of integral equation by Duhamel's integral:
\begin{equation}\label{NS2}
u(t,x)=e^{t\Delta}f(x)-B(u,u)(t,x)\,,
\end{equation}
where the $e^{t\Delta}$ stands for the heat semigroup and $B(u,v)$ is a bilinear operator:
\begin{equation*}
\begin{split}
\widehat{e^{t\Delta}f}(\xi)=&e^{-t\lvert \xi\rvert^2}\hat{f}(\xi)\,;\\
B(u,u)(t,x)=& \int_{0}^{t} e^{(t-s)\Delta}\mathbb {P}\nabla(u\otimes u)ds\,.
\end{split}
\end{equation*}
The mild solution of Navier-Stokes equations was first introduced by T. Kato and H. Fujita in 1962.
Kato-Fujita \cite{KF} considered the well-posedness of Navier-Stokes equations in $\dot{H}^{\frac12}(\mathbb{R}^3)$.
Subsequently, many scholars extended their result to treat well-posedness in a series of function spaces.
For example:
Kato \cite{K} proved well-posedness for solutions in Lebesgue space $L^3$;
Cannone-Meyer-Planchon \cite{CMP} considered in Besov spaces $\dot{B}^{-1+\frac3p,\infty}_{p}$;
Koch-Tataru \cite{KT} generalized the theory of well-posedness to $\textrm{BMO}^{-1}$, which is currently the largest initial space.
\par On the basis of well-posedness, analyticity and the following Gevrey regularity are very significant type of regularity properties of solutions.
If a mild solution $u(t,x)$ also satisfies $e^{(-t\Delta)^\gamma}u\in{ ^{m'}_{m} \dot{F}}^{\frac{n}{p}-1, q}_{p,r}$, 
where $0<\gamma\leq 1$, $e^{(-t\Delta)^\gamma}$ denotes the Fourier multiplier, and the work space ${ ^{m'}_{m} \dot{F}}^{\frac{n}{p}-1, q}_{p,r}$ mentioned will be defined later in Section \ref{sec2}, 
we call the solution of Gevrey regularity.
The solutions of Navier-Stokes equations and many other dissipative equations are actually analytic or Gevrey regular.
Masuda \cite{Ma} was the first to study the analyticity and unique continuation of the Navier-Stokes equation solutions on a bounded domain within Sobolev spaces. 
Later, Foias-Temam \cite{FT} considered periodic functions in Sobolev spaces. 
If solution belongs to their defined $D(A^{\frac12})$ at any time, they obtained the analyticity of solution in time and that the solution also belongs to the Gevrey class $D(A^{\frac12}e^{\sigma A^ {\frac 12}})$. 
In the same year, Giga-Miyakawa \cite{GM} proved the well-posedness and smoothness of global-in-time solution with sufficiently small initial data in Morrey spaces.
Grujic-Kukavica \cite{GK} considered the spatial analyticity of the initial values in the $L_p$ spaces, and they studied the Navier-Stokes equations with complex variables. Biswas-Swanson \cite{BS} proved the analyticity of Fourier coefficients of the periodic solutions in $l_p$ spaces. 
Germain-Pavlović-Staffilani \cite{GPS} focused on small initial values in $\textrm{BMO}^{-1}$ and established the spatial analyticity of solutions through the step-by-step improvement method, strengthening the pointwise regularity of the Koch-Tataru solution in \cite{KT}. 
Dong-Li \cite{DL}, Miura-Sawada \cite{MS} and \cite{GPS} all studied the spatial analyticity by estimating the derivatives. 
Bae-Biswas-Tadmor \cite{BBT} studied the Gevrey regularity in the Besov space $\dot{B} ^{-1+\frac{3}{p},q}_{p}(\mathbb{R}^3)$ under the framework of small initial values 
by Fourier multiplier $e^{t^{\frac12}\Lambda}$, where $e^{t^{\frac12}\Lambda}$ is the multiplier corresponding to $e^{t^{\frac 12}\sum_{i=1}^3|\xi_i|}$. 
Yang \cite{Yang} explored the existence of uniformly analytic solutions in Fourier transform space (Fourier-Herz space). 
One year later, Lou-Yang-He-He \cite{LYHH} presented the existence of uniform analytic solution of fractional Navier-Stokes equations in critical Fourier-Herz space.
Recently, Liu-Zhang \cite{LZ} investigated the global existence of analytical solution of anisotropic Navier-Stokes system within the framework of Besov type functions. 
Here the anisotropic Navier-Stokes system stems from the original Navier-Stokes equations posed on periodic spatial domain. 
They also obtained the precise analytical radius.
These studies, starting from different spaces and methods, gradually deepened the understanding of analyticity theory of the Navier-Stokes equations.

Most of the above studies of Navier-Stokes equations were conducted within the framework of Besov spaces. 
When Fefferman proposed the millennium problem in \cite{Fe},
he emphasized the results on the sigular set of weak solutions of the Navier-Stokes equations.
Caffarelli-Kohn-Nirenberg \cite{CKN} and Lin \cite{Lin} have proved that
the parabolic analogue of Hausdorff dimension of the sigular set sets equal to zero for some suitable weak solutions.
Enlightened by them, we try to control the set of large value points.
Lorentz type spaces, which reflect the distribution of large value points, have a strong connection with blow-up phenomeon.
In view of the great significance of blow-up phenomenon in the regularity theory, 
it is an interesting problem how to establish well-posedness and other properties in the Lorentz type spaces.
Barraza \cite{B} and Meyer-Coifman \cite{MC} proved the global wellposedness for initial data in Lorentz spaces $L^{n,\infty}(\mathbb{R}^n)$.
Lemari\'e gave a simple proof for homogeneous Lorentz space by studying Lebesgues space on the sphere in \cite{Lem}.
Yang-Li \cite{YL} obtained the the global wellposedness in Besov-Lorentz spaces using Hardy-Littlewood maximum operator.
Hobus-Saal got the well-posedness result in Triebel-Lizorkin-Lorentz spaces in \cite{HS} in 2019.
They proved that for the given parameters, Navier-Stokes equations have a unique maximal strong solution.
However, the paper only obtained a local result and cannot extend it to the global case, mainly due to the following reasons: 
Firstly, the global well-posedness relies on the multiplication result in Lemma 6.4 of their article.
It describes the continuity of pointwise multiplication in Triebel-Lizorkin-Lorentz spaces.
This multiplication theorem is not optimal.
In classical spaces such as Lebesgue spaces or Sobolev spaces, there are stronger multiplication theorems that do not require the introduction of small perturbations as in Lemma 6.4 of \cite{HS}.
In the more general Triebel-Lizorkin-Lorentz spaces, the optimal multiplication theorem is not yet available.
Secondly, Hobus and Saal proved the Laplace and the Stokes operator to admit a bounded $H^{\infty}$-calculus and applied it to constructed the maximal strong solution in their work.
Critical spaces are usually key to the global well-posedness,
yet they are rarely involved in \cite{HS}, making it difficult to handle the global estimation.
Overall, Hobus and Saal opened up the first step in the theory of Triebel-Lizorkin-Lorentz spaces for Navier-Stokes equations,
and also left us with the important problem how to consider global properties of the critical Triebel-Lizorkin-Lorentz spaces.

To overcome the above difficulties, we introduce the critical single norm iterative spaces ${^{m'}_{m} \dot{F}}^{\frac{n}{p}-1,q}_{p,r}$ in Section \ref{sec2} via the method of frequency decomposition.
Further, combining the properties of Triebel-Lizorkin-Lorentz spaces, we utilize tools such as the Fefferman-Stein inequality to investigate the properties of our iterative spaces ${ ^{m'}_{m} \dot{F}}^{\frac{n}{p}-1, q}_{p, r}$, which are constructed on the basis of Triebel-Lizorkin-Lorentz spaces.
Through these approaches, first we establish the global well-posedness in critical Triebel-Lizorkin-Lorentz space $\dot{F}^{\frac{n}{p}-1, q}_{p,r}$ absent in Hobus-Saal \cite{HS}.
Moreover, noting that there are relatively few studies on the regularity of solutions in the Lorentz type spaces, 
we achieve the Gevrey regularity result of the global solution.

\begin{theorem}\label{mthmain} (i) Given $1<p,r<\infty$, $m>1$, $1\leq q< \infty$ and $ 0\leq m'<\frac{1}{2}$, resp., $q=\infty$ and $0<m'<\frac12$.
(\ref{NS}) have a unique mild solution in
$({^{m'}_{m} \dot{F}}^{\frac{n}{p}-1,q}_{p,r} )^n$ for all initial data $f(x)$ with $\nabla \cdot f =0$ and $\|f\|_{(\dot{F}^{\frac{n}{p}-1,q}_{p,r})^n}$ small enough.

(ii) Given $  1< r< \infty,n< p< \infty ,m>1-\frac n{2p}$, $0<\gamma<\min\{\frac {m}{2n+2}-\frac 1{4n+4}+\frac n{8pn+8p},\frac 1{4n+4}-\frac n{4pn+4p},\frac m{6n+6}\}$. If the index $(q,m')$ satisfies $$1\leq q<\infty\;\; and \;\; 0\leq m'<\frac12-\frac n{4p}$$ or $$q=\infty \;\;and\;\; 0< m'<\frac12-\frac n{4p},$$
(\ref{NS}) have a unique solution $u(t,x)$ such that $e^{(-t\Delta)^\gamma}u\in({ ^{m'}_{m} \dot{F}}^{\frac{n}{p}-1, q}_{p, r})^n$ for all initial data $u_0(x)$ with $\nabla \cdot u_0 =0$ and $\|u_0\|_{(\dot{F}^{\frac{n}{p}-1,q}_{p,r})^n}$ small enough.
\end{theorem}

Since the proof process of (i) in this theorem is contained in that of (ii), but the proof of regularity theorem is much more complex, here we omit the former.


Although there exist results investigated in spaces based on the Fourier transform, such as \cite{Yang}, which considered the Fourier-Herz spaces.
Many significant regularity studies are carried out within Besov spaces and Lebesgue spaces, such as \cite{BBT} and \cite{LZ}.
Triebel-Lizorkin-Lorentz spaces $\dot{F}^{s,q}_{p,r}$
can be regarded as the real interpolation space of Triebel-Lizorkin space according to \cite{T2} and \cite{YCP}.
They unify a wide array of classical spaces, including Lebesgue spaces $\dot{L}^{n}=\dot{F}^{0,2}_{n,n}$, 
Lorentz spaces $\dot{L}^{n,q}=\dot{F}^{0,2}_{n,q}$, 
Sobolev spaces $\dot{W}^{\frac{n}{p}-1, p}= \dot{F} ^{\frac{n}{p}-1, 2}_{p,p}$
and  Triebel-Lizorkin spaces $\dot{F}^{\frac{n}{p}-1,q}_{p}= \dot{F}^{\frac{n}{p}-1,q}_{p,p}$. 
Our regularity results generalizes the Gevrey estimation of Bae-Biswas-Tadmor \cite{BBT} and extends the regularity studies to the broader framework of Triebel-Lizorkin-Lorentz spaces. 
Furthermore, as mentioned above, Germain-Pavlović-Staffilani \cite{GPS} modified the Koch-Tataru's result to prove that solutions in $\text{BMO}^{-1}$ satisfy derivative estimates, which implies spatial analyticity.
They established regularity for $\text{BMO}^{-1}$, which is the largest known initial space so far, yet their proof is tedious as it proceeds order by order.
In contrast, our Gevrey estimation also contains spatial analyticity since we employs the Fourier multiplier $e^{(-t\Delta)^\gamma}$ to directly capture exponential decay in the frequency domain.
Additionally, Gevrey regularity is more convenient to unify the estimates of gradient of any order due to the rapid attenuation of exponential term.


In the proof of the well-posedness on Besov-Lorentz spaces in \cite{YL}, 
we adopted the classic method of decomposing products into paraproduct flow and coupled flow. 
Here, to ensure that the support of the fourier transform of wavelet functions does not contain the origin point, 
which will be applied when considering priori estimates in Section \ref{sec5}.
Let $Q_{j}v$ be the quantities defined in (\ref{Q}) of Section \ref{mainproof}.
We employ the following decomposition instead:
\begin{equation}\label{eq:decompose}
\begin{array}{rl}
u(t,x)v(t,x)=&\sum_{j-j'\leq-3}Q_juQ_{j'}v+\sum_{\lvert j-j^{\prime}\rvert\leq2}Q_juQ_{j^{\prime}}v+\sum_{j-j'\geq3}Q_juQ_{j'}v.
\end{array}
\end{equation}

\par An outline of this paper is as follows:
In Section \ref{sec2}, we introduce some preliminaries:
Meyer wavelets, Triebel-Lizorkin-Lorentz space, some estimates of maximum operators, and introduce the critical work space.
In Section \ref{sec3}, we discuss some properties for our work space and establish the link between Triebel-Lizorkin-Lorentz spaces and the work space.
In the Section \ref{mainproof}, we transform the proof of our main Theorem to the Theorem \ref{bili}.
Finally,  in Sections \ref{sec5}, \ref{sec6} and \ref{sec7}
we prove Theorem \ref{bili} and establish the Gevrey estimation for small initial data in Triebel-Lizorkin-Lorentz spaces.



\section{Preliminaries}\label{sec2}

In this section, we introduce some preliminary knowledge
relative to Triebel-Lizorkin-Lorentz spaces, wavelets
and the work space ${ ^{m'}_{m} \dot{F}}^{\frac{n}{p}-1, q}_{p, r}$ defined by single norm.
At the end of this section, we present some basic inequalities.

\subsection{Triebel-Lizorkin-Lorentz spaces and Meyer wavelets}

Let $\left\{\varphi_j\right\}_{-\infty}^{\infty}$  be function sequence in $\mathscr{S}\left(\mathbb{R}^n\right)$ that satisfies the following properties.

$$
\begin{gathered}
\operatorname{supp} \hat{\varphi}_j \subset\left\{\xi: 2^{j-1} \leqslant|\xi|\leq2^{j}\right\}, \quad\forall  j\in\mathbb{Z}, \\
|\hat{\varphi}_j(\xi)|\geq C>0, \quad  \; \frac35 \leqslant 2^{-j}|\xi|\leq\frac53, \\
2^{j|\alpha|}\left|D^\alpha \hat{\varphi}_j(\xi)\right| \leqslant C_\alpha, \quad j\in\mathbb{Z}, \quad \forall \xi \in \mathbb{R}^n,\alpha\in\mathbb{N}^n, \\
0< C_1\leq \sum_{j\in\mathbb{Z}} \hat{\varphi}_j(\xi)\leq C_2, \quad \forall \xi \in \mathbb{R}^n .
\end{gathered}
$$

\begin{definition}

{\rm(i)} Given $s\in \mathbb{R}$, $1<p, r< \infty$, $1\leq q\leq \infty$.
$f(x)\in \dot{F}^{s,q}_{p,r} = L^{p,r} (l^{s,q})\Longleftrightarrow$
$$\begin{array}{c}
\sum_{u\in\mathbb{Z}}2^{ur}|\{\sum_{v\in\mathbb{Z}}2^{qv s}|\varphi_v * f|^q>2^{qu}\}|^{\frac{r}{p}}<\infty.
\end{array}$$

{\rm(ii)} Given $s\in \mathbb{R}$ and $1\leq p,q,r\leq \infty$.
$f(x)\in \dot{B}^{s,q}_{p,r} = l^{s,q}(L^{p,r})\Longleftrightarrow$
$$\begin{array}{c}
\sum_{v\in\mathbb{Z}}2^{vqs}(\sum_{u\in\mathbb{Z}}2^{pu}|\{|\varphi_v * f|>2^{u}\}|^{\frac{p}{r}})^{\frac{q}{p}}<\infty.
\end{array}$$

\end{definition}

Let  $g_{s,q} = (\sum_{v\in\mathbb{Z}}2^{qv s}|\varphi_v * f|^q)^{\frac 1q}$. This type of function adopts pointwise control over $f$.
The distribution function restricts the measure of large value points.
 
Further, we note that $g_{s, q}$ is the function obtained by taking $l^{s,q}$ norm to the sequence $\left\{\varphi_j*f\right\}_{-\infty}^{\infty}$.
It is controlled by an Lorentz integral.
This complicates the estimation as we have to handle the distribution of $g_{s, q}$.
It involves pointwise estimation of the maximal function. 
However, we first take the Lorentz space norm when considering Besov Lorentz spaces. Each frequency corresponds to a scalar and no pointwise estimation is required in this case.
 
On the other hand, when we deal with the estimation of Triebel-Lizorkin-Lorentz spaces, we take the $L^{p,r}$ norm by applying the decreasing rearrangement function to $ g_{s,q}$.
Large value points accumulate near the origin after taking the rearrangement function. 
The smaller $r$ is, the more sensitive estimation is to dense large value points, which forces large value points to be sparse.
When $r$ becomes larger, it focuses more on the global supremum.
In the case of Besov-Lorentz spaces, we take the Lorentz index to the convolution sequence, ignoring the pointwise details of $g_{s,q}$.
Therefore, Besov-Lorentz spaces only reflects the average distribution of large value points.
 
In general, the control of Triebel-Lizorkin-Lorentz spaces is more complex, but it captures the distribution of large value points more precisely.
It is more sensitive than the average distribution in the case of Besov-Lorentz spaces and more suitable for the singularity analysis of differential equations.

In addition, Triebel-Lizorkin-Lorentz spaces are the real interpolation spaces of Triebel-Lizorkin spaces.
The introduction of them gives a uniform characterization of many function spaces.
For example, Triebel-Lizorkin-Lorentz spaces cover Bessel-potential spaces, Sobolev-Slobodeckij spaces, etc.
One can find their more properties in \cite{P}, \cite{T1}, \cite{T2} and \cite{YCP}.
Triebel-Lizorkin-Lorentz spaces also have been studied in \cite{HS}, \cite{T1} and \cite{YYH}.
In this paper, we consider also some more properties of this kind of spaces in Sections \ref{sec3}, \ref{sec6} and \ref{sec7} and
apply them to the Gevrey regularity in Section \ref{mainproof}.

Since the discussion on Triebel-Lizorkin-Lorentz spaces is conducted based on their wavelet characterizations, now we introduce Meyer wavelets.
We refer the reader to \cite{Me}, \cite{W}  and \cite{Y} for further information about wavelets. 
Set $E_n=\{0,1\}^n\backslash\{0\}$,
$\Gamma_n=\{(\epsilon,k): \epsilon\in E_n,k\in\mathbb{Z}^n\}$
and $\Lambda_n=\{(\epsilon,j,k): \epsilon\in E_n, j\in \mathbb{Z}, k\in\mathbb{Z}^n\}$.
Let $\phi^0(\xi)$ be a even function in $C^{\infty}_0([-\frac{4\pi}3,\frac{4\pi}3])$ satisfying that
$$\left\{
\begin{aligned}
\phi^0(\xi)=1\; &\; ,\lvert \xi\rvert\leq\frac{2\pi}3;\\ 
0\leq\phi^0(\xi)&\leq1.
\end{aligned}
\right.$$
Set $\varphi(\xi)=[(\phi^0(\frac \xi2))^2-(\phi^0(\xi))^2]^{\frac12}$ and $\phi^1(\xi)=e^{-\frac{i\xi}2}\varphi(\xi)$.
Thus we can get 
$$\left\{
\begin{aligned}
\varphi(\xi)=0\; &\; ,\lvert \xi\rvert\leq\frac{2\pi}3;\\ 
\varphi^2(\xi)+\varphi^2(2\pi-\xi)=1\; &\; ,\frac{2\pi}3\leq \xi\leq\frac{4\pi}3.
\end{aligned}
\right.$$
For arbitrary $\epsilon=(\epsilon_1,\epsilon_2,...,\epsilon_n)\in E_n$,
let $\phi^{\epsilon}(x)$ be the function defined by Fourier transform $\hat\phi^{\epsilon}(\xi)=\Pi_{i=1}^n\phi^{\epsilon_i}(\xi_i)$.
For any $k\in\mathbb{Z}^n$, $j\in\mathbb{Z}$, set $\phi^{\epsilon}_{j,k}(x)=2^{\frac{nj}2}\phi^{\epsilon}(2^jx-k)$.
For the convenience of the later definitions and proofs,
we use $\{\phi^{\epsilon}_{j,k}(x)\}_{(\epsilon,j,k)\in\Lambda_n}$ to denote Meyer wavelets in the rest of this article.

For all $\epsilon\in \{0,1\}^{n}, j\in \mathbb{Z}, k\in \mathbb{Z}^{n}$ and distribution $f(x)$, denote
$f^{\epsilon}_{j,k}=\langle f, \phi^{\epsilon}_{j,k}\rangle$, which we call wavelet coefficients.
Using Meyer wavelets we can characterize $L^2(\mathbb{R}^n)$ in the following way:
\begin{lemma}
The Meyer wavelets make up an orthogonal basis of $L^2(\mathbb{R}^n)$.
Further, for any function $f(x)\in L^2(\mathbb{R}^n)$,
$f(x)=\sum_{(\epsilon,j,k)\in\Lambda_n}f^{\epsilon}_{j,k}\phi^{\epsilon}_{j,k}(x)$ in the $L^2$ convergence sense.
\end{lemma}

For $j\in \mathbb{Z}$,
denote $f_{j}(x) = 2^{\frac{n}{2}j} \sum\limits_{(\epsilon,k)\in \Gamma_n}|f^{\epsilon}_{j,k}|\chi(2^{j}x-k).$
Wavelets can characterize more general function spaces.
The following characterization can be found in \cite{Y} and \cite{YCP}.
\begin{lemma}\label{le:2.2}

{\rm(i)} Given $s\in \mathbb{R}$ and $1\leq p<\infty$,$1\leq q\leq \infty$.
$f(x)\in \dot{F}^{s,q}_{p}\Longleftrightarrow$
$$\begin{array}{c}
[\sum\limits_{\epsilon,j,k}  2^{jq(s+\frac{n}{2})}  | f^{\epsilon}_{j,k} |^{q}\chi(2^jx-k)]^{\frac1q}\in L^p.
\end{array}$$

{\rm(ii)} Given $s\in \mathbb{R}$, $1<p, r< \infty$, $1\leq q\leq \infty$.
$f(x)\in \dot{F}^{s,q}_{p,r} \Longleftrightarrow$
$$\begin{array}{c}
\sum\limits_{u}  2^{ur}  | \{x: \sum\limits_{j} 2^{jsq}|f_{j}(x)|^ q >2^{qu}\}|^{\frac{r}{p}}<\infty.
\end{array}$$
\end{lemma}

{\rm(iii)} Given $s\in \mathbb{R}$ and $1\leq p,q,r\leq \infty$.
$f(x)\in \dot{B}^{s,q}_{p,r}\Longleftrightarrow$
$$\begin{array}{c}
\sum\limits_{j}  2^{jqs}  \{ \sum\limits_{u\in \mathbb{Z}} 2^{u r} |\{x: f_{j}(x)>2^{u}\}|^{\frac{r}{p}}\}^{\frac{q}{r}}<\infty.
\end{array}$$





\subsection{Critical work space}\label{work space}

\par The scale symmetry of Navier-Stokes equations is well-konwn.
Suppose a function $u$ defined on $\mathbb{R}^{n}$ and $u(t,x)$ defined on $(0,\infty)\times\mathbb{R}^{n}$. 
Let $$u_{\theta}(\cdot)= \theta u(\theta \cdot) \;\;\mbox{and}\;\; u_{\theta}(t,x) = \theta u(\theta^{2}t, \theta x).$$
If $u(t,x)$ solves the equations \eqref{NS1}, so does the $ u_{\theta}(t,x)$ (with the corresponding initial condition $u_{\theta}(\cdot)$).
If a function space $X$ satisfies that $\Vert u_{\theta}(\cdot)\Vert_X=\Vert u(\cdot)\Vert_X$.
We call it critical space in Navier-Stokes equations. 

There are many of result works for critical spaces in Navier-Stokes equations, such as:
Cannone \cite{C}, Cannone-Wu \cite{CW}, Li-Xiao-Yang \cite{LXY}, Miura \cite{Mi} and Wu \cite{W1}, \cite{W2}, \cite{W3}.
We can figure out that when $s=\frac {n}{p} -1$, the Triebel-Lizorkin-Lorentz space $\dot{F}^{s,q}_{p,r}$ is critical space. In this paper, we consider critical spaces.
\par Before introducing our work space, we write the following notations for convenience.
For any function $f(t,x)$ defined on $(0,\infty)\times\mathbb{R}^{n}$, let $f^{\epsilon}_{j,k}(t)=\langle f(t,\cdot),\phi^{\epsilon}_{j,k}\rangle$, which is the wavelet coefficients of $f(t,x)$. 
$\forall j, j_t\in \mathbb{Z}$, denote
\begin{align*}
f_{j}(t,x) = &2^{\frac{n}{2}j} \sum\limits_{(\epsilon,k)\in \Gamma}|f^{\epsilon}_{j,k}(t)|\chi(2^{j}x-k);\\
f_{j,j_t}(x)=&\sup_{2^{-2j_t}\leq t<2^{2-2j_t}}f_{j}(t,x).
\end{align*}
For all $p,q,r\geq1$, $m,m'\geq0$,  denote

$$\begin{array}{c}
A_{r,p}^{m,q} =\sup_{j_t\in\mathbb{Z}}\sum_u2^{ur}|\{x:\sum_{j\geq j_t}2^{2(j-j_t)mq}2^{jq(\frac np-1)}(f_{j,j_t}(x))^q>2^{qu}\}|^{\frac rp};\\
A_{r,p}^{m',q} =\sup_{j_t\in\mathbb{Z}}\sum_u2^{ur}|\{x:\sum_{j< j_t}2^{2(j-j_t)m'q}2^{jq(\frac np-1)}(f_{j,j_t}(x))^q>2^{qu}\}|^{\frac rp}.

\end{array}$$


\begin{definition}\label{de:2.3} Given $1\leq p<\infty, 1\leq q,r\leq \infty, m'\geq 0,m>0.$
$f(t,x)\in { ^{m'}_{m} \dot{F}}^{\frac{n}{p}-1, q}_{p, r}$ if and only if
$$\begin{array}{c}
A_{r,p}^{m,q}+A_{r,p}^{m',q}<\infty.
\end{array}$$
\end{definition}

Here we divide ${ ^{m'}_{m} \dot{F}}^{\frac{n}{p}-1, q}_{p, r}$ into high frequency and low frequency according to the scale of $t2^{2j}$.
$m$ represents the regularity and it reflects the rapidly decreasing of the function flow for high frequency.
Different of $m'$ reflects diffrent stability of the trace function, namely the convergence around $t=0$.
$m'$ control the low frequency.
If $m'=0$, then ${ ^{0}_{m} \dot{F}}^{\frac{n}{p}-1, q}_{p, r} \subset L^{\infty} ( \dot{F}^{\frac{n}{p}-1, q}_{p, r}).$
Our single norm space improved the original work space in \cite{HS}.
Here the norm is discrete for both space variable and time.

\subsection{Basic inequalities}

The following inequality can be proved easily:

\begin{lemma}\label{le:2.4}
For $0<r\leq1$ and $a_k\geq0,\;k\in N_+$, we have
$$(\sum_{k>0}a_k)^r\leq \sum_{k>0}(a_k)^r  .$$
\end{lemma}
\begin{proof}
For any $k>0$, let $b_k=(a_k)^r$. Then $$\sum_{k>0}a_k=\sum_{k>0}(b_k)^{\frac 1r}\leq\Vert b\Vert_{\infty}^{\frac 1r-1}\sum_{k>0}b_k\leq(\sum_{k>0}b_k)^{\frac 1r}=[\sum_{k>0}(a_k)^r]^{\frac 1r}  .$$
\end{proof}

In this paper, we have to use some estimations relative to Hardy-Littlewood maximum operator $M$.
For $\{f^{\epsilon}_{j,k}\}_{(\epsilon,j,k)\in \Lambda_n}$,
let $f_{j}(x) = 2^{\frac{n}{2}j} \sum\limits_{(\epsilon,k)\in \Gamma}|f^{\epsilon}_{j,k}|\chi(2^{j}x-k)$ and
\begin{equation*}
\begin{split}
g^k_{j,j'}=\left\{
\begin{array}{ll}
\sum_{(\epsilon',k')\in\Gamma}\frac{2^{\frac{n}{2}j'}|f^{\epsilon'}_{j',k'}|}{(1+|k'-2^{j'-j}k|)^N},&j\geq j',k\in\mathbb{Z}^n;\\
\sum_{(\epsilon',k')\in\Gamma}\frac{2^{\frac{n}{2}j'}|f^{\epsilon'}_{j',k'}|}{(1+|k-2^{j-j'}k'|)^N},&j<j',k\in\mathbb{Z}^n.
\end{array}
\right.
\end{split}
\end{equation*}

Yang \cite [pp. 87-88]{Y} has proved the following Lemma:
\begin{lemma}\label{HL max}
For $N>2n+1$ and $x\in Q_{j,k}$, we have
\begin{equation*}
\begin{split}
g^k_{j,j'}\lesssim\left\{
\begin{array}{ll}
M(f_{j'})(x),&j\geq j'\\2^{n(j'-j)}M(f_{j'})(x),&j<j'
\end{array}
\right.
\end{split}
\end{equation*}
\end{lemma}
\begin{proof}
In fact, we only need to consider the cases where $j\geq j'$ and $j<j'$ respectively. For each given $(j,k)$ and $(j',k')$, we should considering the least dyadic cube containing $Q_{j,k}$ and $Q_{j',k'}$ for maximum operators. We then obtain the lemma after direct calculation. See \cite{Y}, Lemma 3.2, Chapter 5 for details.
\end{proof}

\begin{lemma}\label{FS}( {\texttt Fefferman Stein, See \cite{YCP}} )
For any $1<p,r<\infty$, $1\leq q\leq\infty$, we have
\begin{equation*}
\sum_{u}2^{ur}|\{\sum_{j}(M(f_j)(x)^q)^{\frac 1q}>2^u\}|^{\frac rp}\lesssim\sum_{u}2^{ur}|\{\sum_{j}(f_j(x)^q)^{\frac 1q}>2^u\}|^{\frac rp}.
\end{equation*}
\end{lemma}



\section{Properties of ${ ^{m'}_{m} \dot{F}}^{\frac{n}{p}-1, q}_{p, r}$  and its connection with $\dot{F}^{\frac{n}{p}-1, q}_{p, r}$}\label{sec3}
Let $E$ be the space constructed by $v(t,x)$,
which is equipped with the norm $\Vert v\Vert_{E}=\sup_{t>0}t^{\frac12-\frac n{2q}}\Vert v(t,x)\Vert_{(q,\infty)}+\sup_{t>0}\Vert v(t,x)\Vert_{(n,\infty)}$.
In order to find an accurate semigroup characterization of initial space $\dot{L}^{n,\infty}(\mathbb{R}^n)= \dot{F}^{0, 2}_{n,r}$ and work space $E$,
Barraza \cite{B} obtained exact relationship between his homogeneous Lorentz space $\dot{L}^{n,\infty}(\mathbb{R}^n)$ and work space $E$.
We extend their results to general Triebel-Lizorkin-Lorentz spaces $\dot{F}^{\frac{n}{p}-1, q}_{p,r}$.
But our work space ${ ^{m'}_{m} \dot{F}}^{\frac{n}{p}-1, q}_{p,r}$ is a single norm space,
where $m$ is related to regularity and $m'$ is related to stability.
In this section, we give the corresponding result for initial space $ \dot{F}^{\frac{n}{p}-1, q}_{p,r}$,
which includes not only $L^{n,\infty}(\mathbb{R}^n)$, but also Lebesgue spaces, Sobolev spaces, Triebel-Lizorkin spaces and many other useful spaces.

\subsection{Properties of work space}\label{sec3.1}


Let parameterized Besov space $B_{m,m'}$ be defined as follows:
$$(t2^{2j})^{m} 2^{(\frac{n}{2}-1)j} |f^{\epsilon}_{j,k}(t)|<\infty, \forall t2^{2j}\geq 1.$$
$$(t2^{2j})^{m'} 2^{(\frac{n}{2}-1)j} |f^{\epsilon}_{j,k}(t)|<\infty, \forall 0<t2^{2j}\leq 1.$$
By definition of ${ ^{m'}_{m} \dot{F}}^{\frac{n}{p}-1, q}_{p,r}$,
we have:
\begin{lemma}\label{bmm}
Given $1\leq p<\infty, 1\leq q,r\leq \infty,m'\geq0,m>0,$
$${ ^{m'}_{m} \dot{F}}^{\frac{n}{p}-1, q}_{p,r}\subset B_{m,m'}.$$
\end{lemma}

\begin{proof}
For $E\subset \mathbb{Z}^{n}$, denote $\# E$ the number of elements in $E$.
For any $j_t\in\mathbb{Z}$ and $j\geq j_t$, we have
$$\sum_u2^{ur}|\{x:2^{2(j-j_t)mq}2^{jq(\frac np-1)}(f_{j,j_t}(x))^q>2^{qu}\}|^{\frac rp}<\infty .$$

Hence$$\sum_u2^{ur}|\{x:\sup_{2^{-2j_t}\leq t<2^{2-2j_t}}\sum_{\epsilon,k}|f^{\epsilon}_{j,k}(t)|\chi(2^jx-k)>2^{-\frac n2j}2^{2(j_t-j)m}2^{j(1-\frac np)}2^{u}\}|^{\frac rp}<\infty  .$$
Then we can get that
$$\sum_u2^{ur-nj\frac rp}\#\{k:\sup_{2^{-2j_t}\leq t<2^{2-2j_t}}\sum_{\epsilon}|f^{\epsilon}_{j,k}(t)|>2^{-\frac n2j}2^{2(j_t-j)m}2^{j(1-\frac np)}2^{u}\}^{\frac rp}<\infty  .$$
Let $u_{j,j_t}>0$ and it satisfies that
$$\#\{k:\sup_{2^{-2j_t}\leq t<2^{2-2j_t}}\sum_{\epsilon}|f^{\epsilon}_{j,k}(t)|>2^{-\frac n2j}2^{2(j_t-j)m}2^{j(1-\frac np)}2^{u_{j,j_t}}\}>0$$ and
$$\#\{k:\sup_{2^{-2j_t}\leq t<2^{2-2j_t}}\sum_{\epsilon}|f^{\epsilon}_{j,k}(t)|>2^{-\frac n2j}2^{2(j_t-j)m}2^{j(1-\frac np)}2^{u_{j,j_t}+1}\}=0.$$
We have $2^{ru_{j,j_t}-nj\frac rp}<\infty$ and for any $k\in\mathbb{Z}^n$,
$$\sup_{\epsilon,2^{-2j_t}\leq t<2^{2-2j_t}}|f^{\epsilon}_{j,k}(t)|\leq2^{-\frac n2j}2^{2(j_t-j)m}2^{j(1-\frac np)}2^{u_{j,j_t}+1}.$$
According to the arbitrariness of $j_t$, we obtain
$$(t2^{2j})^{m}2^{-j}\cdot2^{\frac n2 j}\sup\limits_{\epsilon\in E_n}|f^{\epsilon}_{j,k}(t)|<\infty, \forall t2^{2j}\geq1.$$
That means $(t2^{2j})^m2^{j(\frac n2-1)}|f^{\epsilon}_{j,k}(t)|<\infty$.
The same is true for $j<j_t$.

\end{proof}

\subsection{Boundedness of $e^{(-t\Delta)^\gamma}e^{t\Delta}$}
\par For any $f\in\dot{F}^{\frac{n}{p}-1,q}_{p,r}$, let $f(t,x)=e^{(-t\Delta)^\gamma}e^{t\Delta}f$. Take $\{f^{\epsilon}_{j,k}\}_{(\epsilon,j,k)\in\Lambda_n}$ and $\{f^{\epsilon}_{j,k}(t)\}_{(\epsilon,j,k)\in\Lambda_n}$ as their wavelet coefficients.
If $0\leq \gamma<1$,
taking notice of the support set of the Fourier transform of Meyer wavelets, it holds
$$f^{\epsilon}_{j,k}(t)=\sum_{\epsilon^{\prime},\lvert j-j^{\prime}\rvert\leq1,k^{\prime}}f^{\epsilon^{\prime}}_{j^{\prime},k^{\prime}}\\\langle e^{(-t\Delta)^\gamma}e^{t\Delta}\phi^{\epsilon^{\prime}}_{j^{\prime},k^{\prime}},\phi^{\epsilon}_{j,k}\rangle.$$

 For the purpose of precisely clarifying the connection between $f(t,x)$ in ${ ^{m'}_{m} \dot{F}}^{\frac{n}{p}-1, q}_{p,r}$ and its trace function $f(0,x)$ in $ \dot{F} ^{\frac{n}{p}-1,q} _{p,r}$, we consider the continuity of $e^{(-t\Delta)^\gamma}e^{t\Delta}$ in this subsection. For 
$\gamma=0$, this corresponds to the boundedness of heat flow. 

We state the following lemma without proof. The case of $\gamma=0$ has been proved in \cite{LXY}. Other cases can be proved in the same way. 
\begin{lemma}\label{le:3.1}
There exists a constant $N^{\prime}\in\mathbb{N}_+$ large enough and a small constant $\widetilde{c}>0$ that for any positive integer $N$, as long as $N>N^{\prime}$, then
\begin{equation}
\lvert f^{\epsilon}_{j,k}(t)\rvert\lesssim e^{-\widetilde{c}t2^{2j}}\sum_{\epsilon^{\prime},\lvert j-j^{\prime}\rvert\leq 1,k^{\prime}}\lvert f^{\epsilon^{\prime}}_{j^{\prime},k^{\prime}}\rvert(1+\lvert2^{j-j^{\prime}}k^{\prime}-k\rvert)^{-N},\forall t2^{2j}\geq1
\end{equation}
and
\begin{equation}
\lvert f^{\epsilon}_{j,k}(t)\rvert\lesssim \sum_{\epsilon^{\prime},\lvert j-j^{\prime}\rvert\leq 1,k^{\prime}}\lvert f^{\epsilon^{\prime}}_{j^{\prime},k^{\prime}}\rvert(1+\lvert2^{j-j^{\prime}}k^{\prime}-k\rvert)^{-N},\forall 0< t2^{2j}\leq1\,.
\end{equation}
\end{lemma}

\begin{theorem} \label{th:B-to-Y}
Given $0\leq\gamma<1$, $1<p,r< \infty$,$ 1\leq q< \infty$ and $ m'\geq0, m>0$ or $ q=\infty$ and $m',m>0$.
If $f\in \dot{F} ^{\frac{n}{p}-1,q} _{p,r}$, then $ e^{(-t\Delta)^\gamma}e^{t\Delta} f \in { ^{m'}_{m} \dot{F}}^{\frac{n}{p}-1, q}_{p,r}.$

\end{theorem}

\begin{proof}

\par For $t2^{2j}> 1$, by lemma \ref{le:3.1}, we have
$$\lvert f^{\epsilon}_{j,k}(t)\rvert\lesssim\sum_{\epsilon^{\prime},\lvert j-j^{\prime}\rvert\leq1,k^{\prime}}\lvert f^{\epsilon^{\prime}}_{j^{\prime},k^{\prime}}\rvert(1+\lvert 2^{j-j^{\prime}}k^{\prime}-k\rvert)^{-N}e^{-\widetilde{c}t2^{2j}}.$$
We need to use lemma \ref{HL max} and finally get that $$\lvert f_j(t,x)\rvert\lesssim\sum_{\lvert j-j^{\prime}\rvert\leq1} M(f_{j^{\prime}})(x)e^{-\widetilde{c}t2^{2j}}.$$
Case 1: $1\leq q<\infty$. For any $j_t\in\mathbb{Z}$, let $f_{j,j_t}(x)=\sup_{2^{-2j_t}\leq t<2^{2-2j_t}}f_j(t,x)$ and $A_{r,p}^{m,q} =\sup_{j_t\in\mathbb{Z}}\sum_u2^{ur}|\{x:\sum_{j\geq j_t}2^{2(j-j_t)mq}2^{jq(\frac np-1)}(f_{j,j_t}(x))^q>2^{qu}\}|^{\frac rp}$. Then
\begin{equation*}
\begin{split}
A_{r,p}^{m,q} \lesssim&\sup_{j_t\in\mathbb{Z}}\sum_u2^{ur}|\{x:\sum_{j\geq j_t}2^{2(j-j_t)mq}2^{jq(\frac np-1)}e^{-\widetilde{c}q2^{2(j-j_t)}}(\sum_{\lvert j-j^{\prime}\rvert\leq1} M(f_{j^{\prime}})(x))^q>2^{qu}\}|^{\frac rp}\\
\lesssim&\sup_{j_t\in\mathbb{Z}}\sum_u2^{ur}|\{x:\sum_{j\geq j_t}\sum_{\lvert j-j^{\prime}\rvert\leq1}2^{2(j-j_t)mq}e^{-\widetilde{c}q2^{2(j-j_t)}}2^{j'q(\frac np-1)}( M(f_{j^{\prime}})(x))^q>2^{qu}\}|^{\frac rp}\\
\lesssim&\sup_{j_t\in\mathbb{Z}}\sum_u2^{ur}|\{x:\sum_{j'\geq j_t-1}\sum_{j:\lvert j-j^{\prime}\rvert\leq1}2^{j'q(\frac np-1)}( M(f_{j^{\prime}})(x))^q>2^{qu}\}|^{\frac rp}\\
\lesssim&\sum_u2^{ur}|\{x:\sum_{j'}( M(2^{j'(\frac np-1)}f_{j^{\prime}})(x))^q>2^{qu}\}|^{\frac rp}.
\end{split}
\end{equation*}
According to lemma \ref{FS}, we have
\begin{equation*}
\begin{split}
A_{r,p}^{m,q} \lesssim&\sum_u2^{ur}|\{x:\sum_{j'} 2^{j'q(\frac np-1)}(f_{j^{\prime}}(x))^q>2^{qu}\}|^{\frac rp}\lesssim\Vert f\Vert_{{\dot{F}}^{\frac{n}{p}-1, q}_{p,r}}.
\end{split}
\end{equation*}

Case 2: $q=\infty$. Analogously, let $f_{j,j_t}(x)=\sup_{2^{-2j_t}\leq t<2^{2-2j_t}}f_j(t,x)$ and $A_{r,p}^{m,q} =\sup_{j_t\in\mathbb{Z}}\sum_u2^{ur}|\{x:\sup_{j\geq j_t}2^{2(j-j_t)m}2^{j(\frac np-1)}f_{j,j_t}(x)>2^{u}\}|^{\frac rp}$. That denotes
\begin{equation*}
\begin{split}
A_{r,p}^{m,q} \lesssim&\sup_{j_t\in\mathbb{Z}}\sum_u2^{ur}|\{x:\sup_{j\geq j_t}2^{2(j-j_t)m}2^{j(\frac np-1)}e^{-\widetilde{c}2^{2(j-j_t)}}\sum_{\lvert j-j^{\prime}\rvert\leq1} M(f_{j^{\prime}})(x)>2^{u}\}|^{\frac rp}\\
\lesssim&\sup_{j_t\in\mathbb{Z}}\sum_u2^{ur}|\{x:\sup_{j\geq j_t}\sum_{\lvert j-j^{\prime}\rvert\leq1}2^{2(j-j_t)m}e^{-\widetilde{c}2^{2(j-j_t)}}2^{j'(\frac np-1)} M(f_{j^{\prime}})(x)>2^{u}\}|^{\frac rp}\\
\lesssim&\sup_{j_t\in\mathbb{Z}}\sum_u2^{ur}|\{x:\sum_{j'\geq j_t-1}\sup_{j:\lvert j-j^{\prime}\rvert\leq1}2^{2(j-j_t)m}e^{-\widetilde{c}2^{2(j-j_t)}}2^{j'(\frac np-1)} M(f_{j^{\prime}})(x)>2^{u}\}|^{\frac rp}\\
\lesssim&\sum_u2^{ur}|\{x:\sup_{j'} M(2^{j'(\frac np-1)}f_{j^{\prime}})(x)>2^{u}\}|^{\frac rp}\lesssim\Vert f\Vert_{{\dot{F}}^{\frac{n}{p}-1, \infty}_{p,r}}.
\end{split}
\end{equation*}

For $0<t2^{2j}\leq1$, let $A_{r,p}^{m',q} =\sup_{j_t\in\mathbb{Z}}\sum_u2^{ur}|\{x:\sum_{j< j_t}2^{2(j-j_t)m'q}2^{jq(\frac np-1)}(f_{j,j_t}(x))^q>2^{qu}\}|^{\frac rp}$ as we defined in subsection \ref{work space}. In the same way we can get that
\begin{equation*}
\begin{split}
A_{r,p}^{m',q}\lesssim\sum_u2^{ur}|\{x:\sum_{j} 2^{jq(\frac np-1)}(f_{j}(x))^q>2^{qu}\}|^{\frac rp}=\Vert f\Vert_{{\dot{F}}^{\frac{n}{p}-1, q}_{p,r}}\,.
\end{split}
\end{equation*}
\end{proof}




\section{Proof of theorem \ref{mthmain} (ii)}\label{mainproof}
Recall that the integral form of (\ref{NS}):
\begin{equation*}
\begin{split}
u(t,x)=e^{t\Delta}u_0- \int_{0}^{t} e^{(t-s)\Delta}\mathbb {P}\nabla(u\otimes u)ds\,.
\end{split}
\end{equation*}
We aim to find the mild solution $u$ such that $e^{(-t\Delta)^\gamma}u\in({ ^{m'}_{m} \dot{F}}^{\frac{n}{p}-1, q}_{p, r})^n$. Let $\widetilde{u}(t,x)=e^{(-t\Delta)^\gamma}u(t,x)$ and substitute it into integral equation:
\begin{equation*}
\begin{split}
\widetilde{u}(t,x)=e^{(-t\Delta)^\gamma}e^{t\Delta}u_0- e^{(-t\Delta)^\gamma}\int_{0}^{t} e^{(t-s)\Delta}\mathbb {P}\nabla(e^{-(-s\Delta)^\gamma}\widetilde{u}\otimes e^{-(-s\Delta)^\gamma}\widetilde{u})ds\,.
\end{split}
\end{equation*}
We have proved that if $u\in \dot{F}^{\frac{n}{p}-1,q}_{p,r}$, then $e^{(-t\Delta)^\gamma}e^{t\Delta}u\in{ ^{m'}_{m} \dot{F}}^{\frac{n}{p}-1, q}_{p, r}$ for any $0<\gamma<1$ in Section \ref{sec3}. By fixed point theory in \cite{Lem}, it remains to get the boundedness of $B^{\gamma}$, which is defined by
\begin{equation*}
\begin{split}
B^{\gamma}(\widetilde{u},\widetilde{v})=e^{(-t\Delta)^\gamma}\int_{0}^{t} e^{(t-s)\Delta}\mathbb {P}\nabla(e^{-(-s\Delta)^\gamma}\widetilde{u}\otimes e^{-(-s\Delta)^\gamma}\widetilde{v})ds\,.
\end{split}
\end{equation*}
For $l,l',l''=1,\cdots,n$, define
\begin{equation}\label{61eq}
B_{l}(\widetilde{u},\widetilde{v})= e^{(-t\Delta)^\gamma}\int^{t}_{0} e^{(t-s)\Delta}\partial_{x_{l}}(e^{-(-s\Delta)^\gamma}\widetilde{u}\cdot e^{-(-s\Delta)^\gamma}\widetilde{v}) ds,
\end{equation}
\begin{equation}\label{62eq}
B_{l,l',l''}(\widetilde{u},\widetilde{v})= e^{(-t\Delta)^\gamma}\int^{t}_{0} e^{(t-s)\Delta} \partial_{x_l}\partial_{x_{l'}}\partial_{x_{l''}}(-\Delta)^{-1} (e^{-(-s\Delta)^\gamma}\widetilde{u}\cdot e^{-(-s\Delta)^\gamma}\widetilde{v}) ds.
\end{equation}
It's easy to verify that if $B_{l}$ and $B_{l,l',l''}$ is bounded from ${ ^{m'}_{m} \dot{F}}^{\frac{n}{p}-1, q}_{p, r}\times { ^{m'}_{m} \dot{F}}^{\frac{n}{p}-1, q}_{p, r}$
to ${ ^{m'}_{m} \dot{F}}^{\frac{n}{p}-1, q}_{p, r}$ for any $l,l',l''=1,\cdots,n$, then $B^{\gamma}$ is bounded from $({ ^{m'}_{m} \dot{F}}^{\frac{n}{p}-1, q}_{p, r})^{n}\times ({ ^{m'}_{m} \dot{F}}^{\frac{n}{p}-1, q}_{p, r})^{n}$
to $({ ^{m'}_{m} \dot{F}}^{\frac{n}{p}-1, q}_{p, r})^{n}$. Since the term $B_{l}$ is easier to prove, we consider only $B_{l,l',l''}$.

For $j\in\mathbb{Z}$ and $\epsilon\in \{0,1\}^{n}\backslash \{0\}$, define $f^{\epsilon}_{j,k}(t)=\langle f(t),\phi^{\epsilon}_{j,k}\rangle$. Similarly, in the rest of this article, we will denote the wavelet coefficients of $\widetilde{u}(t,x)$ and $\widetilde{v}(t,x)$ by $\widetilde{u}^{\epsilon}_{j,k}(t)$ and $\widetilde{v}^{\epsilon}_{j,k}(t)$ respectively. Let
\begin{equation}\label{Q}
Q_{j}^{\epsilon}f(t,x)= \sum\limits_{k \in \mathbb{Z}^n
} f^{\epsilon}_{j,k}(t) \phi^{\epsilon}_{j,k}(x)
\text{ and } Q_{j}f(x)= \sum\limits_{\epsilon \in
E_n} Q^{\epsilon}_{j} f(x). \end{equation}
We decompose the product of any two functions $\widetilde{u}$
and $\widetilde{v}$ according to the method of (\ref{eq:decompose}).

Define the following three function flows:
\begin{align*}
B_{l,l',l'',1}(\widetilde{u},\widetilde{v})&= e^{(-t\Delta)^\gamma}\int^{t}_{0} e^{(t-s)\Delta} \partial_{x_l}\partial_{x_{l'}}\partial_{x_{l''}}(-\Delta)^{-1} \left(\sum_{\epsilon,\epsilon',\lvert j-j^{\prime}\rvert\leq2}e^{-(-s\Delta)^\gamma}Q^{\epsilon}_j\widetilde{u}\cdot e^{-(-s\Delta)^\gamma}Q^{\epsilon'}_{j^{\prime}}\widetilde{v}\right) ds,\\
B_{l,l',l'',2}(\widetilde{u},\widetilde{v})&=e^{(-t\Delta)^\gamma}\int^{t}_{0} e^{(t-s)\Delta} \partial_{x_l}\partial_{x_{l'}}\partial_{x_{l''}}(-\Delta)^{-1}\left(\sum_{\epsilon,\epsilon'\!,\;j-j'\geq-3}e^{-(-s\Delta)^\gamma}Q^{\epsilon}_j\widetilde{u}\cdot e^{-(-s\Delta)^\gamma}Q^{\epsilon'}_{j'}\widetilde{v}\right) ds,\\
B_{l,l',l'',3}(\widetilde{u},\widetilde{v})&=e^{(-t\Delta)^\gamma}\int^{t}_{0} e^{(t-s)\Delta} \partial_{x_l}\partial_{x_{l'}}\partial_{x_{l''}}(-\Delta)^{-1}\left(\sum_{\epsilon,\epsilon'\!,\;j-j'\leq3}e^{-(-s\Delta)^\gamma}Q^{\epsilon}_j\widetilde{u}\cdot e^{-(-s\Delta)^\gamma}Q^{\epsilon'}_{j'}\widetilde{v}\right) ds
\end{align*}
The continuity of $B_{l,l',l'',1}(\widetilde{u},\widetilde{v})$, $B_{l,l',l'',2}(\widetilde{u},\widetilde{v})$ and $B_{l,l',l'',3}(\widetilde{u},\widetilde{v})$ completes the proof. We summarize the continuity as the following theorem, which are proved in Sections \ref{sec6} and \ref{sec7}. By symmetry, we omit the proof of boundedness of $B_{l,l',l'',3}$.

\begin{theorem} \label{bili} Given  $  1\leq q\leq\infty, 1< r< \infty,n< p< \infty , 0\leq m'<\frac12-\frac n{4p},m>1-\frac n{2p}$, $0<\gamma<\min\{\frac {m}{2n+2}-\frac 1{4n+4}+\frac n{8pn+8p},\frac 1{4n+4}-\frac n{4pn+4p},\frac m{6n+6}\}$. 
$l,l',l''=1,\cdots,n$.
If $\widetilde{u},\widetilde{v}\in { ^{m'}_{m} \dot{F}}^{\frac{n}{p}-1, q}_{p, r}$, then
\begin{align*}
\begin{array}{ll}
 & (i)\;B_{l,l',l'',1}(\widetilde{u},\widetilde{v}) \in { ^{m'}_{m} \dot{F}}^{\frac{n}{p}-1, q}_{p, r}.\\
 & (ii)\;B_{l,l',l'',2}(\widetilde{u},\widetilde{v}), B_{l,l',l'',3}(\widetilde{u},\widetilde{v}) \in { ^{m'}_{m} \dot{F}}^{\frac{n}{p}-1, q}_{p, r}.
\end{array}
\end{align*}
\end{theorem}


\section{ The priori estimates of nonlinear terms}\label{sec5}
In this section, we aim to get the priori estimates related to $B_{l,l',l'',1}(\widetilde{u},\widetilde{v})$ and $B_{l,l',l'',2}(\widetilde{u},\widetilde{v})$. First let us consider $B_{l,l',l'',1}$.
\begin{lemma}\label{B1}
For any $(\epsilon,j,k)\in\Lambda_n$, $l,l',l''=1,\cdots,n\;and\; \epsilon \in \{0,1\}^{n}\backslash \{0\}$, let $f^{\epsilon,1}_{j,k}(t)=\langle B_{l,l',l'', 1}(\widetilde{u},\widetilde{v}), \phi^{\epsilon}_{j, k}\rangle$. Assuming $0<\gamma\leq\frac 12, N\geq0$, there exists the constants $c, \widetilde{c}$ such that if then
\begin{align*}
|f^{\epsilon,1}_{j,k}(t)|\lesssim&\int^{t}_{0}\sum_{j'>j-5}\sum\limits_{j'':|j'-j''|\leq2, \epsilon',\epsilon'', k', k''}\frac{2^{\frac n2 j+j}e^{-c(t-s)2^{2j}}e^{\widetilde{c}(t^{\gamma}2^{2j\gamma}-s^{\gamma}2^{2j\gamma})}}{(1+|2^{j''-j'}k'-k''|)^N(1+|k-2^{j-j'}k'|)^N}\times \\&|\widetilde{u}^{\epsilon'}_{j',k'}(s)| |\widetilde{v}^{\epsilon''}_{j'',k''}(s)|\times\max\{1,(s2^{2j'})^{2\gamma N}\}\times\max\{1,(t2^{2j})^{\gamma N}\}ds.
\end{align*}
\end{lemma}
\begin{proof}
 The support of the Fourier transform of the paraproduct is in the ring.
Note that
$$ {\rm Supp} \widehat{Q^{\epsilon}_{j} \widetilde{u}} \subset
\Big\{\vert\xi_i\vert\leq \frac{4\pi}{3} \cdot 2^{j}, \mbox{ if } \epsilon_i=0;
\frac{2\pi}{3} \cdot 2^{j}\leq
\vert\xi_i\vert \leq \frac{8\pi}{3} \cdot 2^{j}, \mbox{ if } \epsilon_i=1\Big\}.$$
If $|j-j'|\leq1$ and $\epsilon,\epsilon'\in \{0,1\}^{n}\backslash \{0\}$, the support of the Fourier transform of $Q^{\epsilon}_{j}u Q^{\epsilon'}_{j'} v$ is contained in a ring:
$$\Big\{ \vert\xi_i\vert\leq \frac{8\pi}3 \cdot (2^{j}+2^{j'}), \forall i=1,2,3,..,n\Big\}.$$
According to the continuity of $e^{(-t\Delta)^{\gamma}} $ and $A_{l,l',l''}=e^{(t-s)\Delta} \partial_{x_l}\partial_{x_{l'}}\partial_{x_{l''}}(-\Delta)^{-1}$, we have
\begin{align*}
f^{\epsilon,1}_{j,k}(t)=&\left\langle B_{l,l',l'', 1}(\widetilde{u},\widetilde{v}), \phi^{\epsilon}_{j, k}\right\rangle
=\int^{t}_{0}\sum_{j'>j-5}\sum\limits_{|j'-j''|\leq2, \epsilon',\epsilon'', k', k''}\widetilde{u}^{\epsilon'}_{j',k'}(s) \widetilde{v}^{\epsilon''}_{j'',k''}(s)a^{\epsilon',\epsilon''}_{j',j'',k',k''}(s) ds.
\end{align*}
Here $a^{\epsilon',\epsilon''}_{j',j'',k',k''}(s)=\left\langle  A_{l,l',l''} \Big(
 e^{-(-s\Delta)^{\gamma}}\phi^{\epsilon'}_{j',k'}(x) e^{-(-s\Delta)^{\gamma}}\phi^{\epsilon''}_{j'',k''}(x)\Big) , e^{(-t\Delta)^{\gamma}}\phi^{\epsilon}_{j, k}\right\rangle$. Next, let us concentrate on $a^{\epsilon',\epsilon''}_{j',j'',k',k''}(s)$. We can write
 \begin{align*}
&|a^{\epsilon',\epsilon''}_{j',j'',k',k''}(s)|\\=&|\int e^{-(t-s)|\xi|^2} \xi_l\xi_{l'}\xi_{l''}|\xi|^{-2} \Big(
 \widehat{e^{-(-s\Delta)^{\gamma}}\phi^{\epsilon'}_{j',k'}}(\xi) \ast \widehat{e^{-(-s\Delta)^{\gamma}}\phi^{\epsilon''}_{j'',k''}}(\xi)\Big) e^{i 2^{-j}k\xi} e^{t^{\gamma}|\xi|^{2\gamma}}\widehat{\phi^{\epsilon}_{j, 0}}(\xi)d\xi|\\
 =&|\int e^{-(t-s)|\xi|^2} \xi_l\xi_{l'}\xi_{l''}|\xi|^{-2} e^{i 2^{-j}k\xi}e^{t^{\gamma}|\xi|^{2\gamma}}\widehat{\phi^{\epsilon}_{j, 0}}(\xi)
 \times\\&\int
 e^{-s^{\gamma}|\xi-\eta|^{2\gamma}}\widehat{\phi^{\epsilon'}_{j',k'}}(\xi-\eta) e^{-s^{\gamma}|\eta|^{2\gamma}}\widehat{\phi^{\epsilon''}_{j'',k''}}(\eta)d\eta d\xi|\\
 \lesssim&|\int 2^{\frac n2 j+j}e^{-(t-s)2^{2j}|\xi|^2} \xi_l\xi_{l'}\xi_{l''}|\xi|^{-2} e^{t^{\gamma}2^{2j\gamma}|\xi|^{2\gamma}}e^{i\xi (k-2^{j-j'}k')}\widehat{\phi^{\epsilon}}(\xi) d\xi
 \times\\&\int
 e^{-s^{\gamma}|2^{j}\xi-2^{j''}\eta|^{2\gamma}}\widehat{\phi^{\epsilon'}}(2^{j-j'}\xi-2^{j''-j'}\eta) e^{i\eta(2^{j''-j'}k'-k'')} e^{-s^{\gamma}2^{2j''\gamma}|\eta|^{2\gamma}}\widehat{\phi^{\epsilon''}}(\eta)d\eta|.
 \end{align*} 
 We divide the rest of the argument into four cases.\\
 Case 1: $|2^{j''-j'}k'-k''|\leq2$ and $|k-2^{j-j'}k'|\leq2$. Recall that the support of $\widehat{\phi^{\epsilon}}(\xi)$ is contained in a ring. Moreover, for $0<\gamma\leq\frac12$, we have triangle inequality: $$|2^{j}\xi-2^{j''}\eta|^{2\gamma}+|2^{j''}\eta|^{2\gamma}\geq |2^{j}\xi|^{2\gamma},\quad\forall \xi,\eta\in \mathbb{R}^n.$$ A direct computation derives
\begin{align*}
&|a^{\epsilon',\epsilon''}_{j',j'',k',k''}(s)|\lesssim \frac{2^{\frac n2 j+j}e^{-c(t-s)2^{2j}}e^{\widetilde{c}(t^{\gamma}2^{2j\gamma}-s^{\gamma}2^{2j\gamma})}}{(1+|2^{j''-j'}k'-k''|)^N(1+|k-2^{j-j'}k'|)^N}.
\end{align*}
Case 2: $|2^{j''-j'}k'-k''|>2$ and $|k-2^{j-j'}k'|\leq2$. Denote by $l_{i_0}$ the largest component of $2^{j''-j'}k'-k''$. We have $(1+|l_{i_0}|)^N\sim(1+|2^{j''-j'}k'-k''|)^N$, which implies
\begin{align*}
&|a^{\epsilon',\epsilon''}_{j',j'',k',k''}(s)|\\\lesssim&|\int \frac{2^{\frac n2 j+j}e^{-(t-s)2^{2j}|\xi|^2}}{(1+|2^{j''-j'}k'-k''|)^{N}} \xi_l\xi_{l'}\xi_{l''}|\xi|^{-2} e^{t^{\gamma}2^{2j\gamma}|\xi|^{2\gamma}}e^{i\xi (k-2^{j-j'}k')}\widehat{\phi^{\epsilon}}(\xi) d\xi
 \times\\&\int
 e^{-s^{\gamma}|2^{j}\xi-2^{j''}\eta|^{2\gamma}}\widehat{\phi^{\epsilon'}}(2^{j-j'}\xi-2^{j''-j'}\eta) e^{-s^{\gamma}2^{2j''\gamma}|\eta|^{2\gamma}}\widehat{\phi^{\epsilon''}}(\eta)\partial_{\eta_{i_0}}^Ne^{i\eta(2^{j''-j'}k'-k'')} d\eta|.
 \end{align*}
 By integration-by-parts, we obtain that 
 \begin{align*}
&|a^{\epsilon',\epsilon''}_{j',j'',k',k''}(s)|\\\lesssim&|\int \frac{2^{\frac n2 j+j}e^{-(t-s)2^{2j}|\xi|^2}}{(1+|2^{j''-j'}k'-k''|)^{N}} \xi_l\xi_{l'}\xi_{l''}|\xi|^{-2} e^{t^{\gamma}2^{2j\gamma}|\xi|^{2\gamma}}e^{i\xi (k-2^{j-j'}k')}\widehat{\phi^{\epsilon}}(\xi) d\xi
 \times\\&\int
 \sum_{l_1+l_2+l_3=N}\frac{N!\times 2^{(j''-j')l_2}}{l_1!\times l_2!\times l_3!}\partial_{\eta_{i_0}}^{l_1}e^{-s^{\gamma}|2^{j}\xi-2^{j''}\eta|^{2\gamma}-s^{\gamma}2^{2j''\gamma}|\eta|^{2\gamma}} \times\\&\partial_{i_0}^{l_2}\widehat{\phi^{\epsilon'}}(2^{j-j'}\xi-2^{j''-j'}\eta)\partial_{i_0}^{l_3}\widehat{\phi^{\epsilon''}}(\eta)e^{i\eta(2^{j''-j'}k'-k'')} d\eta|.
 \end{align*}
 The support of $\widehat{\phi^{\epsilon}}(\xi)$ is contained in a ring. Hence $$|\partial_{\eta_{i_0}}^{l_1}e^{-s^{\gamma}|2^{j}\xi-2^{j''}\eta|^{2\gamma}-s^{\gamma}2^{2j''\gamma}|\eta|^{2\gamma}}|\lesssim e^{-s^{\gamma}(2^{j}|\xi|)^{2\gamma}}\times\max\{1,(s2^{2j''})^{l_1\gamma}\},$$ which denotes that
 \begin{align*}
&|a^{\epsilon',\epsilon''}_{j',j'',k',k''}(s)|\lesssim \frac{2^{\frac n2 j+j}e^{-c(t-s)2^{2j}}e^{\widetilde{c}(t^{\gamma}2^{2j\gamma}-s^{\gamma}2^{2j\gamma})}}{(1+|2^{j''-j'}k'-k''|)^N(1+|k-2^{j-j'}k'|)^N}\times\max\{1,(s2^{2j'})^{\gamma N}\}.
\end{align*}
Case 3: $|2^{j''-j'}k'-k''|\leq2$ and $|k-2^{j-j'}k'|>2$. Denote by $k_{j_0}$ the largest component of $k-2^{j-j'}k'$. By the same way we use integration-by-parts to get that
\begin{align*}
&|a^{\epsilon',\epsilon''}_{j',j'',k',k''}(s)|\\\lesssim&|\int \frac{2^{\frac n2 j+j}e^{-(t-s)2^{2j}|\xi|^2}}{(1+|k-2^{j-j'}k'|)^{N}} \xi_l\xi_{l'}\xi_{l''}|\xi|^{-2} e^{t^{\gamma}2^{2j\gamma}|\xi|^{2\gamma}}\widehat{\phi^{\epsilon}}(\xi) \partial_{\xi_{j_0}}^Ne^{i\xi (k-2^{j-j'}k')}d\xi
 \times\\&\int
 e^{-s^{\gamma}|2^{j}\xi-2^{j''}\eta|^{2\gamma}}\widehat{\phi^{\epsilon'}}(2^{j-j'}\xi-2^{j''-j'}\eta) e^{-s^{\gamma}2^{2j''\gamma}|\eta|^{2\gamma}}\widehat{\phi^{\epsilon''}}(\eta)e^{i\eta(2^{j''-j'}k'-k'')} d\eta|
 \\\lesssim&\int \sum_{l_1+l_2+l_3+l_4=N}\frac{2^{\frac n2 j+j}2^{l_4(j-j')}}{(1+|k-2^{j-j'}k'|)^{N}} |\partial_{\xi_{j_0}}^{l_1}(e^{-(t-s)2^{2j}|\xi|^2}\xi_l\xi_{l'}\xi_{l''}|\xi|^{-2} e^{t^{\gamma}2^{2j\gamma}|\xi|^{2\gamma}})|d\xi
 \\&\int
|\partial_{j_0}^{l_2}\widehat{\phi^{\epsilon}}(\xi)\partial_{\xi_{j_0}}^{l_3}(e^{-s^{\gamma}|2^{j}\xi-2^{j''}\eta|^{2\gamma}})\partial_{j_0}^{l_4}\widehat{\phi^{\epsilon'}}(2^{j-j'}\xi-2^{j''-j'}\eta) e^{-s^{\gamma}2^{2j''\gamma}|\eta|^{2\gamma}}\widehat{\phi^{\epsilon''}}(\eta)| d\eta.
 \end{align*}
 We have to use the fact that the support of $\widehat{\phi^{\epsilon}}(\xi)$ is contained in a ring and get that there exists a small constant $c_1>0$ satisfying
 \begin{align*}
 |\partial_{\xi_{j_0}}^{l_1}(e^{-(t-s)2^{2j}|\xi|^2}\xi_l\xi_{l'}\xi_{l''}|\xi|^{-2} e^{t^{\gamma}2^{2j\gamma}|\xi|^{2\gamma}})|&\lesssim e^{-c_1(t-s)2^{2j}|\xi|^2}e^{t^{\gamma}2^{2j\gamma}|\xi|^{2\gamma}} \max\{1,(t2^{2j})^{\gamma l_1}\},\\
 |\partial_{\xi_{j_0}}^{l_3}(e^{-s^{\gamma}|2^{j}\xi-2^{j''}\eta|^{2\gamma}})|&\lesssim e^{-s^{\gamma}|2^{j}\xi-2^{j''}\eta|^{2\gamma}}\max\{1,(s2^{2j''})^{\gamma l_3}\} .
 \end{align*}
 Consequently, we have the constants $c, \widetilde{c}>0$ such that
 \begin{align*}
&|a^{\epsilon',\epsilon''}_{j',j'',k',k''}(s)|\\\lesssim& \frac{2^{\frac n2 j+j}e^{-c(t-s)2^{2j}}e^{\widetilde{c}(t^{\gamma}2^{2j\gamma}-s^{\gamma}2^{2j\gamma})}}{(1+|2^{j''-j'}k'-k''|)^N(1+|k-2^{j-j'}k'|)^N}\times\max\{1,(s2^{2j'})^{\gamma N}\}\times\max\{1,(t2^{2j})^{\gamma N}\}.
\end{align*}
Case 4: $|2^{j''-j'}k'-k''|>2$ and $|k-2^{j-j'}k'|>2$. In a similar manner to treat Case 2 and Case 3. We obtain
\begin{align*}
&|a^{\epsilon',\epsilon''}_{j',j'',k',k''}(s)|\\\lesssim&|\int \frac{2^{\frac n2 j+j}e^{-(t-s)2^{2j}|\xi|^2}}{(1+|k-2^{j-j'}k'|)^{N}(1+|2^{j''-j'}k'-k''|)^N} \xi_l\xi_{l'}\xi_{l''}|\xi|^{-2} e^{t^{\gamma}2^{2j\gamma}|\xi|^{2\gamma}} \partial_{\xi_{j_0}}^Ne^{i\xi (k-2^{j-j'}k')}d\xi
 \\&\int\widehat{\phi^{\epsilon}}(\xi) 
 e^{-s^{\gamma}|2^{j}\xi-2^{j''}\eta|^{2\gamma}}\widehat{\phi^{\epsilon'}}(2^{j-j'}\xi-2^{j''-j'}\eta) e^{-s^{\gamma}2^{2j''\gamma}|\eta|^{2\gamma}}\widehat{\phi^{\epsilon''}}(\eta)\partial_{\eta_{i_0}}^Ne^{i\eta(2^{j''-j'}k'-k'')} d\eta|.
 \end{align*}
 As in Case 2 and Case 3, applying integration-by-parts twice and utilizing the fact that the support of $\widehat{\phi^{\epsilon}}(\xi)$ is a ring again, we finally get that
 \begin{align*}
&|a^{\epsilon',\epsilon''}_{j',j'',k',k''}(s)|\\\lesssim& \frac{2^{\frac n2 j+j}e^{-c(t-s)2^{2j}}e^{\widetilde{c}(t^{\gamma}2^{2j\gamma}-s^{\gamma}2^{2j\gamma})}}{(1+|2^{j''-j'}k'-k''|)^N(1+|k-2^{j-j'}k'|)^N}\times\max\{1,(s2^{2j'})^{2\gamma N}\}\times\max\{1,(t2^{2j})^{\gamma N}\}.
\end{align*}
This gives the estimate of $a^{\epsilon',\epsilon''}_{j',j'',k',k''}(s)$. We substitute it into $f^{\epsilon,1}_{j,k}(t)$ and then complete the proof.
\end{proof}
By the same method we obtain the following estimates for $B_{l,l',l'',2}$.  
\begin{lemma}\label{B2}
For any $(\epsilon,j,k)\in\Lambda_n$, $l,l',l''=1,\cdots,n$, let $f^{\epsilon,2}_{j,k}(t)=\langle B_{l,l',l'', 2}(\widetilde{u},\widetilde{v}), \phi^{\epsilon}_{j, k}\rangle$. Assuming $0<\gamma\leq\frac 12, N>0$, there exists the constants $c, \widetilde{c}$ such that
\begin{align*}
|f^{\epsilon,2}_{j,k}(t)|\lesssim&\int^{t}_{0}\sum_{j':|j-j'|\leq2}\sum\limits_{j''\leq j'-3, \epsilon',\epsilon'', k', k''}\frac{2^{\frac n2 j''+j}e^{-c(t-s)2^{2j}}e^{\widetilde{c}(t^{\gamma}2^{2j\gamma}-s^{\gamma}2^{2j\gamma})}}{(1+|2^{j''-j'}k'-k''|)^N(1+|k-2^{j-j'}k'|)^N}\times \\&|\widetilde{u}^{\epsilon'}_{j',k'}(s)| |\widetilde{v}^{\epsilon''}_{j'',k''}(s)|\times\max\{1,(t2^{2j})^{\gamma N}\}\max\{1,(s2^{2j''})^{\gamma N}\}ds.
\end{align*}
\end{lemma}


\section{ The proof of theorem \ref{bili}(i)}\label{sec6}

In Sections \ref{sec6} and \ref{sec7}, we only consider $1\leq q<\infty$. Since that the same method can be applied to the case of $q=\infty$. First, let us consider the indices where $0<t2^{2j}\leq 1$. By lemma \ref{B1}, we have
\begin{align*}
|f^{\epsilon,1}_{j,k}(t)|\lesssim&\int^{t}_{0}\sum_{j'>j-5}\sum\limits_{j'':|j'-j''|\leq2, \epsilon',\epsilon'', k', k''}\frac{2^{\frac n2 j+j}}{(1+|2^{j''-j'}k'-k''|)^N(1+|k-2^{j-j'}k'|)^N}\times \\&|\widetilde{u}^{\epsilon'}_{j',k'}(s)| |\widetilde{v}^{\epsilon''}_{j'',k''}(s)|\times\max\{1,(s2^{2j'})^{2\gamma N}\}ds.
\end{align*}

In the remaining part of this article, we always set $N=2n+2$. Write the integral as the sum of two terms: $\int_{0}^{2^{-2j'}}$ and $\int_{2^{-2j'}}^{t}$. Denote them by $f^{\epsilon,1,1}_{j,k}(t)$ and $f^{\epsilon,1,2}_{j,k}(t)$ respectively. Since that we can deal with the case of $0<t2^{2j'}< 1$ in the similar manner, here we only consider $t2^{2j'}\geq 1$, namely, $j'\geq j_t-1$.
 
For the first term, we have $0<s \leq 2^{-2j'}$ and $ |v^{\epsilon''}_{j'',k''}(s)|\lesssim(s2^{2j''})^{-m'} 2^{(1-\frac{n}{2})j''}$. Which derives that
\begin{align*}
|f^{\epsilon,1,1}_{j,k}(t)|\lesssim&\sum_{j'\geq j_t-1}\int^{2^{-2j'}}_{0}\sum\limits_{j'':|j'-j''|\leq2, \epsilon',k', k''}\frac{2^{\frac n2 j+j}}{(1+|2^{j''-j'}k'-k''|)^N(1+|k-2^{j-j'}k'|)^N}\times \\&|\widetilde{u}^{\epsilon'}_{j',k'}(s)| (s2^{2j'})^{-m'} 2^{(1-\frac{n}{2})j'}ds\\
\lesssim&\sum_{j'\geq j_t-1}\int^{2^{-2j'}}_{0}\sum\limits_{ \epsilon', k'}\frac{2^{\frac n2 (j-j')+j+j'}}{(1+|k-2^{j-j'}k'|)^N}\times|\widetilde{u}^{\epsilon'}_{j',k'}(s)| (s2^{2j'})^{-m'} ds.
\end{align*}

For any $j',j_s\in\mathbb{Z}$, let $\widetilde{u}_{j'}(s,x) = 2^{\frac{n}{2}j'} \sum\limits_{(\epsilon',k')\in \Gamma}|\widetilde{u}^{\epsilon'}_{j',k'}(s)|\chi(2^{j'}x-k')$, $\widetilde{u}_{j',j_s}(x)=\sup_{2^{-2j_s}\leq s<2^{2-2j_s}} \widetilde{u}_{j'}(s,x)$, $(\widetilde{u}^{\epsilon'}_{j',k'})_{j_s}=\sup_{2^{-2j_s}\leq s<2^{2-2j_s}} |\widetilde{u}^{\epsilon'}_{j',k'}(s)|$. According to lemma \ref{HL max} we obtain
\begin{align*}
f^{1,1}_{j,j_t}(x)=&\sup_{2^{-2j_t}\leq t<2^{2-2j_t}}f^{1,1}_{j}(t,x)=\sup_{2^{-2j_t}\leq t<2^{2-2j_t}}2^{\frac{n}{2}j} \sum\limits_{(\epsilon,k)\in \Gamma}|f^{\epsilon,1,1}_{j,k}(t)|\chi(2^{j}x-k)\\
\lesssim &   \sup_{2^{-2j_t}\leq t<2^{2-2j_t}} \sum\limits_{(\epsilon,k)\in \Gamma} \sum\limits_{j'\geq  j_t-1 } \sum_{j_s\geq j'+1}\int^{2^{2-2j_s}}_{2^{-2j_s}} 2^{n(j-j')+j+j'}(s2^{2j'})^{-m'}
 ds\\&\times\sum\limits_{\epsilon',k'}\frac{2^{\frac n2j'} (\widetilde{u}^{\epsilon'}_{j',k'})_{j_s}}
{\left(1+ \left\vert2^{j-j'}k'-k\right\vert\right)^{N} }\chi(2^{j}x-k)\\
\lesssim &   \sup_{2^{-2j_t}\leq t<2^{2-2j_t}} \sum\limits_{(\epsilon,k)\in \Gamma} \sum\limits_{j'\geq  j_t-1 } \sum_{j_s\geq j'+1}\int^{2^{2-2j_s}}_{2^{-2j_s}} 2^{n(j-j')+j+j'}(s2^{2j'})^{-m'}
 ds\\&\times 2^{n(j'-j)}M(\widetilde{u}_{j',j_s})\chi(2^{j}x-k)\\
\lesssim &    \sum\limits_{j'\geq  j_t-1 } \sum_{j_s\geq j'+1}2^{j-j'}2^{2(j'-j_s)(1-m')} M(\widetilde{u}_{j',j_s}).
\end{align*}

Define $$A_{r,p}^{m',q,1,1} =\sup_{j_t\in\mathbb{Z}}\sum_u2^{ur}|\{x:\sum_{j< j_t}2^{2(j-j_t)m'q}2^{jq(\frac np-1)}(f^{1,1}_{j,j_t}(x))^q>2^{qu}\}|^{\frac rp}.$$ That is to say, for $r\leq p$, according to lemma \ref{le:2.4} and \ref{FS} we have
\begin{align*}
A_{r,p}^{m',q,1,1} \lesssim&\sup_{j_t\in\mathbb{Z}}\sum_u2^{ur}|\{x:\sum_{j< j_t}\sum\limits_{j'\geq  j_t-1 } \sum_{j_s\geq j'+1}2^{2(j-j_t)m'q}2^{jq(\frac np-1)}2^{q(j-j')}2^{2q(j'-j_s)(1-m')}\times\\& M(\widetilde{u}_{j',j_s})^q2^{\delta q(j_s-j')}>2^{qu}\}|^{\frac rp}\\
\lesssim&\sup_{j_t\in\mathbb{Z}}\sum_u2^{ur}(\sum\limits_{j_s\geq  j_t}|\{x:\sum_{j< j_t} \sum_{j_t-1\leq j'\leq j_s-1}2^{2(j-j_t)m'q}2^{jq(\frac np-1)}2^{q(j-j')}\times\\& M(\widetilde{u}_{j',j_s})^q2^{\delta q(j_s-j')+\delta'(j_s-j_t)}2^{2q(j'-j_s)(1-m')}>2^{qu}\}|)^{\frac rp}\\
\lesssim&\sup_{j_t\in\mathbb{Z}}\sum_u2^{ur}(\sum\limits_{j_s\geq  j_t}|\{x: \sum_{j_t-1\leq j'\leq j_s-1}2^{j_tq(\frac np-1)}2^{q(j_t-j')}2^{2q(j'-j_s)(1-m')} \times\\&M(\widetilde{u}_{j',j_s})^q2^{\delta q(j_s-j')+\delta'(j_s-j_t)}>2^{qu}\}|)^{\frac rp}\\
\lesssim&\sup_{j_s\in\mathbb{Z}}\sum_u2^{ur}|\{x:\sum_{j'\leq j_s}\widetilde{u}_{j',j_s}(x)^q2^{2m'q(j'-j_s)}2^{j'q(\frac np-1)}>2^{qu}\}|^{\frac rp}.
\end{align*}
Here we need $0<\delta'<\frac{qn}p$ and $0<\delta q<2q-4m'q-\delta'$.

For $r> p$, we can also obtain the same result through H\"older inequality. We omit this part.

For the second term, according to lemma \ref{bmm}, $s2^{2j'}>1$ gives that $ |\widetilde{v}^{\epsilon''}_{j'',k''}(s)|\lesssim(s2^{2j''})^{-m} 2^{(1-\frac{n}{2})j''}$.  We obtain the estimation of $f^{1,2}_{j,j_t}(x)$ in the same way:
\begin{align*}
f^{1,2}_{j,j_t}(x)\lesssim    \sum\limits_{j'\geq  j_t-1 } \sum_{j_t\leq j_s\leq j'}2^{j-j'}2^{2(j'-j_s)(1+2\gamma N-m)} M(\widetilde{u}_{j',j_s}).
\end{align*}

Define $$A_{r,p}^{m',q,1,2} =\sup_{j_t\in\mathbb{Z}}\sum_u2^{ur}|\{x:\sum_{j< j_t}2^{2(j-j_t)m'q}2^{jq(\frac np-1)}(f^{1,2}_{j,j_t}(x))^q>2^{qu}\}|^{\frac rp}.$$
We shall adopt the same procedure as in the proof of $A_{r,p}^{m',q,1,2}$. Since $ 0<\gamma<\frac{n+4mp-2p}{4Np}$, there exist $\delta$ and $\delta'$ such that $0<\delta'<\frac{nq}p-\delta q$ and $0<\delta\leq \frac np+4m-2-4N\gamma$. In which case we can control it with $\widetilde{u}^{\epsilon'}_{j',k'}(s)$:
\begin{align*}
A_{r,p}^{m',q,1,2}\lesssim&\;_{j_t\in\mathbb{Z}}\sum_u2^{ur}|\{x:\sum_{j< j_t}2^{2(j-j_t)m'q}2^{jq(\frac np-1)}\times\\&( \sum\limits_{j'\geq  j_t-1 } \sum_{j_t\leq j_s\leq j'}2^{j-j'}2^{2(j'-j_s)(1+2\gamma N-m)} M(\widetilde{u}_{j',j_s}))^q>2^{qu}\}|^{\frac rp}\\
\lesssim&\;\sup_{j_t\in\mathbb{Z}}\sum_u2^{ur}\sum_{ j_s\geq j_t}|\{x:\sum\limits_{j'\geq  j_s } 2^{j_tq(\frac np-1)} 2^{q(j_t-j')}2^{2q(j'-j_s)(1+2\gamma N-m)}\times\\&2^{\delta q(j'-j_t)+\delta'(j_s-j_t)} M(\widetilde{u}_{j',j_s})^q>2^{qu}\}|^{\frac rp}\\\lesssim&\;\sup_{j_s\in\mathbb{Z}}\sum_u2^{ur}|\{x:\sum_{j'\geq j_s}\widetilde{u}_{j',j_s}(x)^q2^{2mq(j'-j_s)}2^{j'q(\frac np-1)}>2^{qu}\}|^{\frac rp}.
\end{align*}

For the simplicity, here we omit the case of $r>p$.

Now we aim to consider the case of $t2^{2j}>1$. Lemma \ref{B1} gives that
\begin{align*}
|f^{\epsilon,1}_{j,k}(t)|\lesssim&\int^{t}_{0}\sum_{j'>j-5}\sum\limits_{j'':|j'-j''|\leq2, \epsilon',\epsilon'', k', k''}\frac{2^{\frac n2 j+j}e^{-c(t-s)2^{2j}}e^{\widetilde{c}(t^{\gamma}2^{2j\gamma}-s^{\gamma}2^{2j\gamma})}(t2^{2j})^{\gamma N}}{(1+|2^{j''-j'}k'-k''|)^N(1+|k-2^{j-j'}k'|)^N}\times \\&|\widetilde{u}^{\epsilon'}_{j',k'}(s)| |\widetilde{v}^{\epsilon''}_{j'',k''}(s)|\times\max\{1,(s2^{2j'})^{2\gamma N}\}ds.
\end{align*}
 Here we assume that $2^{-2j'}< \frac t2$, namely, $t2^{2j}> 2^9$, since that proof for $1<t2^{2j}\leq 2^9$ is easy. After dividing the integral into $\int^{2^{-2j'}}_{0}$ and $\int^{t}_{2^{-2j'}}$, we denote the two terms by $f^{\epsilon,1,3}_{j,k}(t)$ and $f^{\epsilon,1,4}_{j,k}(t)$ respectively.
 
We first consider the estimation of $f^{\epsilon,1,3}_{j,k}(t)$. Notice that in this case there is only one more term $e^{-c(t-s)2^{2j}}e^{\widetilde{c}(t^{\gamma}2^{2j\gamma}-s^{\gamma}2^{2j\gamma})}(t2^{2j})^{\gamma N}$ than in the case $0<t2^{2j}\leq 1$. Using the same method, we get that there exists a constant $c'>0$ satisfying
\begin{align*}
f^{1,3}_{j,j_t}(x)=&\sup_{2^{-2j_t}\leq t<2^{2-2j_t}}2^{\frac{n}{2}j} \sum\limits_{(\epsilon,k)\in \Gamma}|f^{\epsilon,1,3}_{j,k}(t)|\chi(2^{j}x-k)\\
\lesssim &    \sum\limits_{j'\geq  j_t-1 } \sum_{j_s\geq j'+1}2^{j-j'}2^{2(j'-j_s)(1-m')}2^{2(j-j_t)\gamma N}e^{-c'2^{2(j-j_t)}} M(\widetilde{u}_{j',j_s}).
\end{align*}
 
Define $$A_{r,p}^{m,q,1,3} =\sup_{j_t\in\mathbb{Z}}\sum_u2^{ur}|\{x:\sum_{j\geq j_t}2^{2(j-j_t)mq}2^{jq(\frac np-1)}(f^{1,3}_{j,j_t}(x))^q>2^{qu}\}|^{\frac rp}.$$
Denote $0<\delta< \min\{2q-4m'q, \frac {qn}p\}$. It is sufficient to get that
\begin{align*}
A_{r,p}^{m,q,1,3}\lesssim&\sup_{j_t\in\mathbb{Z}}\sum_u2^{ur}|\{x:\sum_{j\geq j_t}2^{2(j-j_t)mq}2^{jq(\frac np-1)}\times\\&( \sum\limits_{j'\geq  j_t-1 } \sum_{j_s\geq j'+1}2^{j-j'}2^{2(j'-j_s)(1-m')}2^{2(j-j_t)\gamma N}e^{-c'2^{2(j-j_t)}} M(\widetilde{u}_{j',j_s}))^q>2^{qu}\}|^{\frac rp}\\
\lesssim&\sup_{j_t\in\mathbb{Z}}\sum_u2^{ur}\sum_{j_s\geq j_t}|\{x:\sum_{j\geq j_t}\sum\limits_{j_t-1\leq j'\leq  j_s-1 }2^{2(j-j_t)mq}2^{jq(\frac np-1)}  2^{q(j-j')}\times\\&2^{2q(j'-j_s)(1-m')}2^{2q(j-j_t)\gamma N}e^{-c'q2^{2(j-j_t)}} M(\widetilde{u}_{j',j_s})^q2^{\delta (j_s-j_t)}>2^{qu}\}|^{\frac rp}\\
\lesssim&\sup_{j_t\in\mathbb{Z}}\sum_u2^{ur}\sum_{j_s\geq j_t}|\{x:\sum\limits_{j_t-1\leq j'\leq  j_s-1 }2^{j_tq(\frac np-1)}  2^{q(j_t-j')}2^{2q(j'-j_s)(1-m')}\times\\& M(\widetilde{u}_{j',j_s})^q2^{\delta (j_s-j_t)}>2^{qu}\}|^{\frac rp}\\
\lesssim&\sup_{j_s\in\mathbb{Z}}\sum_u2^{ur}|\{x:\sum_{j'\leq j_s}\widetilde{u}_{j',j_s}(x)^q2^{2m'q(j'-j_s)}2^{j'q(\frac np-1)}>2^{qu}\}|^{\frac rp}.
\end{align*}
For the simplicity, we may take $r\leq p$.
 
For the term of $f^{\epsilon,1,4}_{j,k}(t)$, we also omit the situation of $r> p$. An argument similar to the one used in the proof of $f^{\epsilon,1,2}_{j,k}(t)$ shows that
\begin{align*}
f^{1,4}_{j,j_t}(x)\lesssim&    \sup_{2^{-2j_t}\leq t<2^{2-2j_t}}2^{\frac{n}{2}j} \sum\limits_{(\epsilon,k)\in \Gamma}\int^{t}_{2^{-2j'}}\sum_{j'>j-5}\sum\limits_{j'':|j'-j''|\leq2, \epsilon', k'}|\widetilde{u}^{\epsilon'}_{j',k'}(s)| (s2^{2j''})^{-m} 2^{(1-\frac{n}{2})j''}\times \\&\frac{2^{\frac n2 j+j}e^{-c(t-s)2^{2j}}e^{\widetilde{c}(t^{\gamma}2^{2j\gamma}-s^{\gamma}2^{2j\gamma})}(t2^{2j})^{\gamma N}}{(1+|k-2^{j-j'}k'|)^N}\times\max\{1,(s2^{2j'})^{2\gamma N}\}ds\chi(2^{j}x-k).
\end{align*}
In order to make the estimation more precise, we have to divide the integral into $\int_{2^{-2j'}}^{\frac t2}$ and $\int^{t}_{\frac t2}$. Dividing the $\int_{2^{-2j'}}^{\frac t2}$ into dyadic intervals, there exists a positive constant $c'$ such that
\begin{align*}
f^{1,4}_{j,j_t}(x)\lesssim&    \sup_{2^{-2j_t}\leq t<2^{2-2j_t}}2^{\frac{n}{2}j} \sum_{j'>j-5}(\sum_{j_t\leq j_s\leq j'}\int^{2^{2-2j_s}}_{2^{-2j_s}}+\int^{t}_{\frac t2})M(\widetilde{u}_{j'})2^{-\frac{nj'}2}2^{n(j'-j)}\times \\&2^{\frac n2 j+j}e^{-c(t-s)2^{2j}}e^{\widetilde{c}(t^{\gamma}2^{2j\gamma}-s^{\gamma}2^{2j\gamma})}(t2^{2j})^{\gamma N} (s2^{2j'})^{-m} 2^{(1-\frac{n}{2})j'}(s2^{2j'})^{2\gamma N}ds\\
\lesssim&    \sum_{j'>j-5}\sum_{j_t\leq j_s\leq j'}M(\widetilde{u}_{j',j_s})e^{-c'2^{2(j-j_t)}} 2^{2(j'-j_s)(2\gamma N-m)} 2^{j+j'-2j_s}+\\
&\sum_{j'>j-5}\max_{j_s=j_t\; or \; j_t+1}M(\widetilde{u}_{j',j_s})2^{j'-j}2^{2(j-j_t)\gamma N} 2^{2(j'-j_s)(2\gamma N-m)} .
\end{align*}

Define $$A_{r,p}^{m,q,1,4} =\sup_{j_t\in\mathbb{Z}}\sum_u2^{ur}|\{x:\sum_{j\geq j_t}2^{2(j-j_t)mq}2^{jq(\frac np-1)}(f^{1,4}_{j,j_t}(x))^q>2^{qu}\}|^{\frac rp}.$$
Then
\begin{align*}
A_{r,p}^{m,q,1,4} \lesssim&\sup_{j_t\in\mathbb{Z}}\sum_u2^{ur}|\{x:\sum_{j\geq j_t}2^{2(j-j_t)mq}2^{jq(\frac np-1)}(\sum_{j'>j-5}\sum_{j_t\leq j_s\leq j'}M(\widetilde{u}_{j',j_s})e^{-c'2^{2(j-j_t)}} \times \\&2^{2(j'-j_s)(2\gamma N-m)} 2^{j+j'-2j_s})^q>2^{qu}\}|^{\frac rp}+\\
&\sup_{j_t\in\mathbb{Z}}\sum_u2^{ur}|\{x:\sum_{j\geq j_t}2^{2(j-j_t)mq}2^{jq(\frac np-1)}(\sum_{j'>j-5}\max_{j_s=j_t\; or \; j_t+1}M(\widetilde{u}_{j',j_s})2^{j'-j}\times \\&2^{2(j-j_t)\gamma N} 2^{2(j'-j_s)(2\gamma N-m)} )^q>2^{qu}\}|^{\frac rp}=M_1+M_2.
\end{align*}
We separate the two parts of the above inequality and recall that $0<\gamma<\min\{\frac n{4pN}-\frac 1{2N}+\frac mN,\frac m{3N}\}$. For $M_1$, there exists a constant $0<\delta<\min\{\frac {n}{p},\frac np-2+2m-4\gamma N\}$ satisfying that
\begin{align*}
M_1\!\!\lesssim&\sup_{j_t\in\mathbb{Z}}\sum_u2^{ur}\sum_{j_s\geq j_t}|\{x:\sum_{j\geq j_t}\sum_{j'\geq j_s}2^{2(j-j_t)mq}2^{jq(\frac np-1)}M(\widetilde{u}_{j',j_s})^qe^{-qc'2^{2(j-j_t)}} \times \\&2^{2q(j'-j_s)(2\gamma N-m)} 2^{q(j+j'-2j_s)}(j'-j_t)^q2^{\delta q(j'-j)}>2^{qu}\}|^{\frac rp}\\
\lesssim&\sup_{j_t\in\mathbb{Z}}\sum_u2^{ur}\sum_{j_s\geq j_t}|\{x:\sum_{j'\geq j_s}2^{j_tq(\frac np-1)}M(\widetilde{u}_{j',j_s})^q 2^{2q(j'-j_s)(2\gamma N-m)} 2^{q(j_t+j'-2j_s)}\times \\&(j'-j_t)^q2^{\delta q(j'-j_t)}>2^{qu}\}|^{\frac rp}\\
\lesssim&\sup_{j_s\in\mathbb{Z}}\sum_u2^{ur}|\{x:\sum_{j'\geq j_s}\widetilde{u}_{j',j_s}(x)^q2^{2mq(j'-j_s)}2^{j'q(\frac np-1)}>2^{qu}\}|^{\frac rp}.
\end{align*}
The estimation of $M_2$ completes the proof of theorem \ref{bili} (i). For simplicity, we assume that $j_s=j_t$.
\begin{flalign*}
\begin{split}
M_2\lesssim&\sup_{j_s\in\mathbb{Z}}\sum_u2^{ur}|\{x:\sum_{j'\geq j_s-4}\sum_{j_s\leq j\leq j'+4}2^{2(j-j_s)mq}2^{jq(\frac np-1)}M(\widetilde{u}_{j',j_s})^q2^{q(j'-j)}2^{2q(j-j_s)\gamma N} \times \\& 2^{2q(j'-j_s)(2\gamma N-m)} 2^{\delta q(j'-j)}>2^{qu}\}|^{\frac rp}\\
\lesssim&\sup_{j_s\in\mathbb{Z}}\sum_u2^{ur}|\{x:\sum_{j'\geq j_s}\widetilde{u}_{j',j_s}(x)^q2^{2mq(j'-j_s)}2^{j'q(\frac np-1)}>2^{qu}\}|^{\frac rp}.
\end{split}&
\end{flalign*}
Here we need $0<\delta<\frac n{p}-2+ 2m$.


\section{ The proof of theorem \ref{bili}(ii)}\label{sec7}



We only prove the boundedness of $B^{\epsilon,\epsilon'}_{l,l',l'',2}(u,v)$.

First, Assume $0<t2^{2j}\leq 1$. Denote $f^{\epsilon,2}_{j,k}(t)$ in this case by $f^{\epsilon,2,1}_{j,k}(t)$. Then we have $0<s2^{2j'}\leq 4$ and $ 0<s2^{2j''}\leq 1$. According to lemma \ref{bmm}, it implies that $ |v^{\epsilon''}_{j'',k''}(s)|\lesssim(s2^{2j''})^{-m'} 2^{(1-\frac{n}{2})j''}$. Thus 
\begin{align*}
|f^{\epsilon,2,1}_{j,k}(t)|\lesssim&\int^{t}_{0}\sum_{j':|j-j'|\leq2}\sum\limits_{j''\leq j'-3, \epsilon',\epsilon'', k', k''}\frac{2^{\frac n2 j''+j}}{(1+|2^{j''-j'}k'-k''|)^N(1+|k-2^{j-j'}k'|)^N}\times \\&|\widetilde{u}^{\epsilon'}_{j',k'}(s)| (s2^{2j''})^{-m'} 2^{(1-\frac{n}{2})j''}ds.
\end{align*}
Observe that $0\leq m'\leq \frac 12$. Since $N=2n+2$, we can get that 
\begin{align*}
|f^{\epsilon,2,1}_{j,k}(t)|\lesssim&\sum_{j':|j-j'|\leq2}\int^{t}_{0}\sum\limits_{ \epsilon', k'}\frac{2^{(1-2m')j'+j}|\widetilde{u}^{\epsilon'}_{j',k'}(s)| s^{-m'}}{(1+|k-2^{j-j'}k'|)^N} ds.
\end{align*}
Compare the estimation of $f^{\epsilon,2,1}_{j,k}(t)$ and $f^{\epsilon,1,1}_{j,k}(t)$, which get in Section \ref{sec6}. Suitable modification to the proof for $f^{\epsilon,1,1}_{j,k}(t)$ can show that 
\begin{align*}
f^{2,1}_{j,j_t}(x)=&\sup_{2^{-2j_t}\leq t<2^{2-2j_t}}2^{\frac{n}{2}j} \sum\limits_{(\epsilon,k)\in \Gamma}|f^{\epsilon,2,1}_{j,k}(t)|\chi(2^{j}x-k)\\
\lesssim  &   \sum\limits_{j':|j-j'|\leq2 } \sum_{j_s\geq j_t}2^{2(j'-j_s)(1-m')} M(\widetilde{u}_{j',j_s}).
\end{align*}
Just like in Section \ref{sec6}, we can using H\"older inequality to deal with the case of $r\geq p$. Here we only consider $r<p$.
 
Define $$A_{r,p}^{m',q,2,1} =\sup_{j_t\in\mathbb{Z}}\sum_u2^{ur}|\{x:\sum_{j< j_t}2^{2(j-j_t)m'q}2^{jq(\frac np-1)}(f^{2,1}_{j,j_t}(x))^q>2^{qu}\}|^{\frac rp}.$$
Substitute the estimation of $f^{2,1}_{j,j_t}$ into the definition above. After direct calculation we get
\begin{align*}
A_{r,p}^{m',q,2,1} \lesssim&\sup_{j_t\in\mathbb{Z}}\sum_u2^{ur}\sum_{j_s\geq j_t}|\{x:\sum_{j< j_t}\sum\limits_{j':|j-j'|\leq2 } 2^{2(j-j_t)m'q}2^{jq(\frac np-1)} 2^{2q(j'-j_s)(1-m')}\times \\& M(\widetilde{u}_{j',j_s})^q2^{\delta q(j_s-j_t)}>2^{qu}\}|^{\frac rp}.
\end{align*}
Here we set $0<\delta<2-4m'$.

It follows from $m'<\frac 1{2}$ that 
\begin{align*}
A_{r,p}^{m',q,1,1} 
\lesssim&\sup_{j_s\in\mathbb{Z}}\sum_u2^{ur}|\{x:\sum_{j'\leq j_s}\widetilde{u}_{j',j_s}(x)^q2^{2m'q(j'-j_s)}2^{j'q(\frac np-1)}>2^{qu}\}|^{\frac rp}\\=&\Vert \widetilde{u}\Vert_{{ ^{m'}_{m} \dot{F}}^{\frac{n}{p}-1, q}_{p, r}}.
\end{align*}

All that remains is to consider the case of $t2^{2j}>1$. Denote $f^{\epsilon,2}_{j,k}(t)$ by $f^{\epsilon,2,2}_{j,k}(t)$ in this case. Set $\displaystyle m_{j,s}=\left\{\begin{array}{l}m,\ s2^{2j}>1\\ m',\ 0<s2^{2j}\leq1\end{array}\right.$. According to lemmas \ref{B2} and \ref{bmm}, we have\begin{align*}
|f^{\epsilon,2,2}_{j,k}(t)|\lesssim&\sum_{j''\leq j-3}\sum\limits_{j':|j-j'|\leq2, \epsilon',\epsilon'', k', k''}\frac{2^{\frac {n}2 (j''-j')+(2-2m_{j',s})j'}(t2^{2j})^{\gamma N}}{(1+|2^{j''-j'}k'-k''|)^{-N}(1+|k-2^{j-j'}k'|)^{-N}}\times \\&\int^{t}_{0}e^{-c(t-s)2^{2j}}e^{\widetilde{c}(t^{\gamma}2^{2j\gamma}-s^{\gamma}2^{2j\gamma})}s^{-m_{j',s}} |\widetilde{v}^{\epsilon''}_{j'',k''}(s)|\max\{1,(s2^{2j''})^{\gamma N}\}ds.
\end{align*}
Without loss of generality, we may assume $j=j'$. Other cases can be proved by the same method as employed in $j=j'$. Using lemma \ref{HL max}, we get that
\begin{align*}
f^{2,2}_{j,j_t}(x)\lesssim&\sup_{2^{-2j_t}\leq t<2^{2-2j_t}}\sum_{j''\leq j-1}2^{(2-2m_{j,s})j}(t2^{2j})^{\gamma N}\times \\&\int^{t}_{0}M(\widetilde{v}_{j''})e^{-c(t-s)2^{2j}}e^{\widetilde{c}(t^{\gamma}2^{2j\gamma}-s^{\gamma}2^{2j\gamma})}s^{-m_{j,s}} \max\{1,(s2^{2j''})^{\gamma N}\}ds.
\end{align*}

Define $$A_{r,p}^{m,q,2,2} =\sup_{j_t\in\mathbb{Z}}\sum_u2^{ur}|\{x:\sum_{j\geq j_t}2^{2(j-j_t)mq}2^{jq(\frac np-1)}(f^{2,2}_{j,j_t}(x))^q>2^{qu}\}|^{\frac rp}.$$
Since $0<\gamma<\frac 1{2N}-\frac n{2pN}$, let $0<\delta<\min\{1-\frac np-2m',1-\frac np-2\gamma N\}$, H\"older inequality gives that
 \begin{align*}
\!\!\!A_{r,p}^{m,q,2,2} \lesssim&\sup_{j_t\in\mathbb{Z}}\sum_u2^{ur}|\{x:\sum_{j\geq j_t}2^{2(j-j_t)mq}2^{jq(\frac np-1)}\!\!\!\!\!\!\!\!\!\!\sup_{2^{-2j_t}\leq t<2^{2-2j_t}}\!\!\sum_{j''\leq j-1}2^{2q(1-m_{j,s})j+\delta q(j-j'')}(t2^{2j})^{q\gamma N} \\&(\int^{t}_{0}M(\widetilde{v}_{j''})e^{-c(t-s)2^{2j}}e^{\widetilde{c}(t^{\gamma}2^{2j\gamma}-s^{\gamma}2^{2j\gamma})}s^{-m_{j,s}} \max\{1,(s2^{2j''})^{\gamma N}\}ds)^q>2^{qu}\}|^{\frac rp}.
\end{align*}
 Now we need to decompose the right hand-side above. Write it as the sum of the following seven terms. Set $I_1=I_4=[0,2^{-2j}]$, $I_2=[2^{-2j},\frac t2]$, $I_3=I_7=[\frac t2,t]$, $I_5=[2^{-2j},2^{-2j''}]$, $I_6=[2^{-2j''},\frac t2]$, then
\begin{align*}
&A_{r,p}^{m,q,2,2}\\
\lesssim&\sum_{i=1,2,3}\sup_{j_t\in\mathbb{Z}}\sum_u2^{ur}|\{x:\sum_{j\geq j_t}2^{2(j-j_t)mq}2^{jq(\frac np-1)}\sup_{2^{-2j_t}\leq t<2^{2-2j_t}}\sum_{j''\leq j_t}2^{2q(1-m_{j,s})j+\delta q(j-j'')}\times \\&(t2^{2j})^{q\gamma N}(\int_{I_i}M(\widetilde{v}_{j''})e^{-c(t-s)2^{2j}}e^{\widetilde{c}(t^{\gamma}2^{2j\gamma}-s^{\gamma}2^{2j\gamma})}s^{-m_{j,s}} ds)^q>2^{qu}\}|^{\frac rp}+\\&\sum_{i=4,5,6,7}
\sup_{j_t\in\mathbb{Z}}\sum_u2^{ur}|\{x:\sum_{j\geq j_t}2^{2(j-j_t)mq}2^{jq(\frac np-1)}\!\!\!\!\!\!\!\!\!\!\sup_{2^{-2j_t}\leq t<2^{2-2j_t}}\!\!\sum_{j_t<j''\leq j-1}2^{2q(1-m_{j,s})j+\delta q(j-j'')}\times \\&(t2^{2j})^{q\gamma N}(\int_{I_i}M(\widetilde{v}_{j''})e^{-c(t-s)2^{2j}}e^{\widetilde{c}(t^{\gamma}2^{2j\gamma}-s^{\gamma}2^{2j\gamma})}s^{-m_{j,s}} \max\{1,(s2^{2j''})^{\gamma N}\}ds)^q>2^{qu}\}|^{\frac rp}\\=&\sum_{i=1,2,3}M_i+\sum_{i=4,5,6,7}M_i.
\end{align*}

Set $0<\delta'<2-4m'$. As to $M_1$, decompose the integral into dyadic interval and apply H\"older inequality. Direct calculations derive that there exists a positive constant $c'$ satisfying 
\begin{align*}
M_1\lesssim&\sup_{j_t\in\mathbb{Z}}\sum_u2^{ur}|\{x:\sum_{j\geq j_t}2^{2(j-j_t)mq}2^{jq(\frac np-1)}\sum_{j''\leq j_t}2^{2q(1-m')j+\delta q(j-j'')}2^{2q\gamma N(j-j_t)}\times \\&\sum_{j_s>j}2^{q\delta'(j_s-j)-2qj_s}M(\widetilde{v}_{j'',j_s})^qe^{-c'q2^{2(j-j_t)}}2^{2qm'j_s} >2^{qu}\}|^{\frac rp}.
\end{align*}
Put $0<\delta''<2q-4m'q-q\delta'$. Change the order of sum and we finally obtain
\begin{align*}
M_1\lesssim&\sup_{j_t\in\mathbb{Z}}\sum_u2^{ur}\sum_{j_s>j_t}|\{x:\!\!\!\!\sum_{j_t\leq j< j_s}\sum_{j''\leq j_t}2^{2(j-j_t)mq}2^{jq(\frac np-1)}2^{2q(1-m')j+\delta q(j-j'')}2^{2q\gamma N(j-j_t)} \\&2^{q\delta'(j_s-j)-2qj_s+\delta''(j_s-j_t)}M(\widetilde{v}_{j'',j_s})^qe^{-c'q2^{2(j-j_t)}}2^{2qm'j_s} >2^{qu}\}|^{\frac rp}\lesssim\Vert \widetilde{v}\Vert_{{ ^{m'}_{m} \dot{F}}^{\frac{n}{p}-1, q}_{p, r}}.
\end{align*}
Similarly, for $2\leq i\leq7$, we get
\begin{align*}
M_2\lesssim&\sup_{j_t\in\mathbb{Z}}\sum_u2^{ur}|\{x:\sum_{j\geq j_t}2^{2(j-j_t)mq}2^{jq(\frac np-1)}\sum_{j''\leq j_t}2^{2q(1-m)j+\delta q(j-j'')}2^{2q\gamma N(j-j_t)}\times \\&(\sum_{j_t\leq j_s\leq j}(j-j_t)^q2^{-2qj_s(1-m)}M(\widetilde{v}_{j'',j_s})^qe^{-c'q2^{2(j-j_t)}}>2^{qu}\}|^{\frac rp};\\
M_3\lesssim&\sup_{j_s\in\mathbb{Z}}\sum_u2^{ur}|\{x:\sum_{j\geq j_s}2^{2(j-j_s)mq}2^{jq(\frac np-1)}\sum_{j''\leq j_s}2^{2q(1-m)j+\delta q(j-j'')}2^{2q\gamma N(j-j_s)}\times \\&M(\widetilde{v}_{j'',j_s})^q2^{-2qj}2^{2mqj_s} >2^{qu}\}|^{\frac rp};\\
M_4\lesssim&\sup_{j_t\in\mathbb{Z}}\sum_u2^{ur}|\{x:\sum_{j\geq j_t}2^{2(j-j_t)mq}2^{jq(\frac np-1)}\sum_{j_t<j''\leq j-1}2^{2q(1-m')j+\delta q(j-j'')}2^{2q\gamma N(j-j_t)}\times \\&\sum_{j_s>j}2^{\delta' q(j_s-j)}M(\widetilde{v}_{j'',j_s})^qe^{-c'q2^{2(j-j_t)}}2^{2j_sm'q-2j_sq} >2^{qu}\}|^{\frac rp};\\
M_5\lesssim&\sup_{j_t\in\mathbb{Z}}\sum_u2^{ur}|\{x:\sum_{j\geq j_t}2^{2(j-j_t)mq}2^{jq(\frac np-1)}\sum_{j_t<j''\leq j-1}2^{2q(1-m)j+\delta q(j-j'')}2^{2q\gamma N(j-j_t)}\times \\&\sum_{j''< j_s\leq j}(j-j'')^qM(\widetilde{v}_{j'',j_s})^qe^{-c'q2^{2(j-j_t)}}2^{2qj_sm-2qj_s}>2^{qu}\}|^{\frac rp};\\
M_6\lesssim&\sup_{j_t\in\mathbb{Z}}\sum_u2^{ur}|\{x:\sum_{j\geq j_t}2^{2(j-j_t)mq}2^{jq(\frac np-1)}\sum_{j_t<j''\leq j-1}2^{2q(1-m)j+\delta q(j-j'')}2^{2q\gamma N(j-j_t)}\times \\&\sum_{j_t\leq j_s\leq j''}(j''-j_t)^qM(\widetilde{v}_{j'',j_s})^qe^{-c'q2^{2(j-j_t)}}2^{2qj_sm-2qj_s} 2^{2q\gamma N(j''-j_s)}ds>2^{qu}\}|^{\frac rp};\\
M_7\lesssim&\sup_{j_s\in\mathbb{Z}}\sum_u2^{ur}|\{x:\sum_{j\geq j_s}2^{2(j-j_s)mq}2^{jq(\frac np-1)}\sum_{j_s<j''\leq j-1}2^{2q(1-m)j+\delta q(j-j'')}2^{2q\gamma N(j-j_s)}\times \\&M(\widetilde{v}_{j'',j_s})^q2^{-2qj}2^{2qj_sm} 2^{2q\gamma N(j''-j_s)}>2^{qu}\}|^{\frac rp}.
\end{align*}
In addition, since that $0<\gamma< \frac {m}{2N}$. It's evident to see that
\begin{align*}
M_i\leq \Vert \widetilde{v}\Vert_{{ ^{m'}_{m} \dot{F}}^{\frac{n}{p}-1, q}_{p,r}}, \quad 2\leq i\leq 7
\end{align*}
holds.



{\bf This research did not receive any specific grant from funding agencies in the public, commercial, or not-for-profit sectors.}












\end{document}